\documentclass[11pt,reqno]{amsart}
\title{Entropy-efficient finitary codings}
\date{\today}

\author{Tom Meyerovitch}
\address{Tom Meyerovitch\hfill\break\indent
	Ben Gurion University of the Negev.
	Departement of Mathematics.
	Be'er Sheva, 8410501, Israel.
}
\email{mtom@bgu.ac.il}

\author{Yinon Spinka}
\address{Yinon Spinka\hfill\break\indent
	University of British Columbia.
    Department of Mathematics.
 	Vancouver, BC V6T 1Z2, Canada.\break\indent
 	Tel Aviv University. School of Mathematical Sciences. Tel Aviv 6997801, Israel.}
    \email{yinon@math.ubc.ca}

\usepackage{amsmath, amsthm, amsopn, amssymb, microtype}
\usepackage{enumerate, tikz, etoolbox, intcalc, geometry, caption, subcaption, afterpage}
\usepackage[english]{babel}

\usepackage[colorlinks=true,urlcolor=blue,pdfborder={0 0 0}]{hyperref}
\hypersetup{linkcolor=[rgb]{0,0,0.6}}
\hypersetup{citecolor=[rgb]{0,0.6,0}}

\usepackage[nameinlink]{cleveref}
  \crefname{theorem}{Theorem}{Theorems}
  \crefname{thm}{Theorem}{Theorems}
  \crefname{mainthm}{Theorem}{Theorems}
  \crefname{lemma}{Lemma}{Lemmas}
  \crefname{lem}{Lemma}{Lemmas}
  \crefname{remark}{Remark}{Remarks}
  \crefname{prop}{Proposition}{Propositions}
  \crefname{defn}{Definition}{Definitions}
  \crefname{corollary}{Corollary}{Corollaries}
  \crefname{cor}{Corollary}{Corollaries}
  \crefname{section}{Section}{Sections}
  \crefname{figure}{Figure}{Figures}
  \crefname{quest}{Question}{Questions}

\newtheorem{thm}{Theorem}[section]
\newtheorem{lemma}[thm]{Lemma}
\newtheorem{prop}[thm]{Proposition}

\newtheorem{cor}[thm]{Corollary}

\newtheorem{quest}[thm]{Question}

\theoremstyle{definition}
\newtheorem{remark}{Remark}

\newcommand{\cA}{\mathcal{A}}

\newcommand{\cC}{\mathcal{C}}

\newcommand{\cB}{\mathcal{B}}

\newcommand{\cL}{\mathcal{L}}

\newcommand{\cF}{\mathcal{F}}
\newcommand{\cK}{\mathcal{K}}

\newcommand{\cS}{\mathcal{S}}
\newcommand{\cT}{\mathcal{T}}
\newcommand{\cI}{\mathcal{I}}

\newcommand{\cX}{\mathcal{X}}
\newcommand{\cY}{\mathcal{Y}}

\newcommand{\N}{\mathbb{N}}

\newcommand{\Z}{\mathbb{Z}}
\newcommand{\Q}{\mathbb{Q}}

\newcommand{\V}{\mathbb{V}}
\newcommand{\E}{\mathbb{E}}

\renewcommand{\Pr}{\mathbb{P}}
\newcommand{\1}{\mathbf{1}}

\newcommand{\iid}{i.i.d.}


\begin{document}

\begin{abstract}
We show that any finite-entropy, countable-valued finitary factor of an \iid\ process can also be expressed as a finitary factor of a finite-valued \iid\ process whose entropy is arbitrarily close to the target process. 
As an application, we give an affirmative answer to a question of van den Berg and Steif~\cite{van1999existence} about the critical Ising model on $\Z^d$.
En route, we prove several results about finitary isomorphisms and finitary factors. Our results are developed in a new framework for processes invariant to a permutation group of a countable set satisfying specific properties. This new framework includes all ``classical'' processes over countable amenable groups and all invariant processes on transitive amenable graphs with ``uniquely centered balls''. Some of our results are new already for $\Z$-processes. We prove a relative version of Smorodinsky's isomorphism theorem for finitely dependent $\Z$-processes. We also extend the Keane--Smorodinsky finitary isomorphism theorem to countable-valued \iid\ processes and to \iid\ processes taking values in a Polish space.
\end{abstract}

\maketitle

\section{Introduction}\label{sec:introduction}

In the the late 1990s, van den Berg and Steif~\cite{van1999existence} showed that the Ising model on $\Z^d$ is a finitary factor of an \iid\ process if and only if it admits a unique Gibbs measure (which is now known to be if and only if the temperature is at least the critical temperature).
They further showed that the unique Gibbs measure is a finitary factor of a \emph{finite-valued} \iid\ process throughout the entire high-temperature regime, and they asked whether this is also the case at criticality~\cite[Question~1]{van1999existence}.
 We provide a general result which gives an affirmative answer to this question, and which further shows that the finite-valued \iid\ process can be chosen to have low entropy (arbitrarily close to that of the Ising Gibbs measure), both at criticality and in the high-temperature regime.

\begin{thm}\label{thm:amenable_groups_efficient_finitary_factors_of_iid}
	Let $\Gamma$ be a countable amenable group and let $X$ be a finite-valued $\Gamma$-process which is a finitary factor of an \iid\ process.
	Then, for any $\epsilon>0$, $X$ is a finitary $\Gamma$-factor of a finite-valued \iid\ process whose entropy is at most $h(X)+\epsilon$. 
\end{thm}


To the best of our knowledge, already the case $\Gamma = \mathbb{Z}$ of \cref{thm:amenable_groups_efficient_finitary_factors_of_iid} is new.
We deduce \cref{thm:amenable_groups_efficient_finitary_factors_of_iid} as a special case of a more general result (\cref{thm:main}), which in the case of the Ising model on $\Z^d$ implies that the asserted finitary factor can furthermore by chosen to be equivariant with respect to reflections and rotations by $\frac{\pi}{2}$.

It has been conjectured for some time that (for $\mathbb{Z}$-processes) any finitary factor of an \iid\ process is finitarily isomorphic to an \iid\ process \cite{rudolph1981characterization,smorodinsky1992finitary}. This would have immediately implied the case $\Gamma=\mathbb{Z}$ of \cref{thm:amenable_groups_efficient_finitary_factors_of_iid}. However, a recent result of Gabor \cite{gabor2019finitary} refutes this conjecture, 
so that the conclusion of  \cref{thm:amenable_groups_efficient_finitary_factors_of_iid} cannot in general be strengthened to yield that $X$ is finitarily $\Gamma$-isomorphic to an \iid\ process, even in the simplest case $\Gamma=\mathbb{Z}$.

Our main result is stated in terms of processes which are invariant with respect to  a group of permutations acting on the countable set which is the domain of the process: Let $\V$ be a countable set and let  $\Gamma < \mathit{Perm}(\V)$ be a subgroup of the permutations of $\V$. A \textbf{$(\V,\Gamma)$-process} is a random function from $\V$ to some  Polish space $\cA$, whose distribution is invariant under the action of $\Gamma$. Equivalently, a  $(\V,\Gamma)$-process is a $\Gamma$-invariant probability measure on $\cA^\V$. 
We say that $\Gamma$ is \textbf{nice} if it acts transitively on $\V$, the orbit of every $v \in \V$ under the stabilizer of any $w \in \V$ is finite, it is unimodular, and it satisfies a certain ``aperiodicity condition'', whose precise definition we postpone to \cref{sec:nice-groups}. In general, we do not assume that the group $\Gamma$ is countable.

Two important settings captured by this definition are the following:
\begin{itemize}
	\item $\V$ is itself a countable group, and $\Gamma$ is the group of permutations corresponding to left multiplication. This is equivalent to the action of $\Gamma$ on $\V$ being transitive and free. In this case, $(\V,\Gamma)$-processes are naturally identified with $\Gamma$-processes in the classical sense.
	\item $\V$ is the vertex set of a locally finite, connected, vertex-transitive, unimodular graph $G$ which has uniquely centered balls (see \cref{sec:nice-groups}), and $\Gamma$ is the group of automorphisms of the graph $G$ (i.e., permutations of $\V$ that preserve the edges of $G$).
\end{itemize}
The framework of nice permutation groups thus extends the classical setting of $\Gamma$-processes, and the following is a generalization of \cref{thm:amenable_groups_efficient_finitary_factors_of_iid} to this framework.

\begin{thm}\label{thm:main}
	Let $\V$ be a countable set and let $\Gamma < \mathit{Perm}(\V)$ be a nice amenable group. Let $X$ be a countable-valued $(\V,\Gamma)$-process which is a finitary $\Gamma$-factor of an \iid\ process. 
	Then, for any $\epsilon>0$, there exists a process $X'$ whose single-site entropy $H(X'_{v})$ is at most $\epsilon$ such that $(X,X')$ is finitarily $\Gamma$-isomorphic to a countable-valued \iid\ process. Moreover, when $X$ has finite entropy, the latter \iid\ process can be taken to be finite-valued, and in particular, $X$ is a finitary $\Gamma$-factor of a finite-valued \iid\ process whose entropy is at most $h(X)+\epsilon$.
\end{thm}

The entropy of a $(\V,\Gamma)$-process $X$ over a nice amenable group $\Gamma$ is a quantity denoted by $h(X)$. As we show in \cref{sec:nice-groups}, $h(X)$ is an isomorphism invariant for $(\V,\Gamma)$-processes over nice amenable groups, which coincides with the Kolmogorov--Sinai entropy in the classical setting of $\Gamma$-processes over countable amenable groups. As in the classical case, $h(X)$ is monotone under factors, is bounded above by the single-site entropy $H(X_v)$, and equals the latter for \iid\ processes.

We point out that besides the more general framework, \cref{thm:main} offers two further improvements over \cref{thm:amenable_groups_efficient_finitary_factors_of_iid}: It allows processes with countably infinite alphabets, and it gives a certain finitary isomorphism result.

\medbreak

Our initial motivation for introducing  $(\V,\Gamma)$-processes was to obtain results about finitary factor maps which are equivariant with respect to the group of automorphisms of certain graphs (e.g., when $\V=\Z$ and $\Gamma$ is the permutation group generated by translations and reflections).
However, once the above framework is introduced, it is natural to ask for which groups of permutations $\Gamma < \mathit{Perm}(\V)$ is entropy
a complete invariant for isomorphism of \iid\ $(\V,\Gamma)$-processes, meaning that two such processes are isomorphic if and only if they have equal entropy.
This problem consists of two distinct parts, necessity and sufficiency. Our main results are closely related to the ``sufficiency'' direction, which we proceed to discuss. We defer a brief discussion of the ``necessity'' part to \cref{sec:conclusion}.

Following Stepin \cite{stepin1976bernoulli}, we say that a group $\Gamma$ of permutations of a countable set $\V$ is \textbf{Ornstein} if any two equal-entropy \iid\ $(\V,\Gamma)$-processes are isomorphic, and \textbf{finitarily Ornstein} if any two countable-valued\footnote{The ``countable-valued'' assumption is needed since a finitary factor of a countable-valued process is always countable valued as well, so that a countable-valued process cannot be finitarily isomorphic to an uncountable-valued process.} equal-entropy \iid\ $(\V,\Gamma)$ processes are finitarily isomorphic.
There is the question of equal entropy being a sufficient condition for (finitary) isomorphism: 

\begin{quest}\label{quest:entropy_complete}
	Which permutation groups $\Gamma$ are (finitarily) Ornstein?
\end{quest}

We note that each of the two properties is preserved when passing to a subgroup.
Ornstein's isomorphism theorem and Keane--Smorodinsky~\cite{keane1979bernoulli} answer the primary case where $\Gamma$ is a cyclic group generated by a transitive permutation of $\V$.
In the case of a permutation group $\Gamma$ whose action is transitive and free, Seward~\cite{seward2018bernoulli} proved that $\Gamma$ is Ornstein, and that any two \emph{finite-valued} equal-entropy \iid\ $(\V,\Gamma)$-processes are finitarily isomorphic.
Seward's result was preceded by a partial result of Bowen \cite{bowen2012almostornstien} which showed that any such group is \textbf{almost Ornstein}, meaning that any two equal-entropy \iid\ $(\V,\Gamma)$-processes are isomorphic, as long as neither of the two processes has a single-site marginal that is supported on a two element set.  As a byproduct of the proof of our main result, we obtain a strengthening of Bowen's result, showing that nice permutation groups are ``almost finitarily Ornstein'':

\begin{thm}\label{thm:iid-isomorphic}
	Let $\V$ be a countable set and let $\Gamma < \mathit{Perm}(\V)$ be a nice permutation group. Then any two countable-valued equal-entropy \iid\ $(\V,\Gamma)$-processes taking more than two values are finitarily $\Gamma$-isomorphic.
\end{thm}

Already in the case of ordinary $\Gamma$-processes, 
 \cref{thm:iid-isomorphic} is not fully contained in Seward's result because our result also yields finitary isomorphism of equal-entropy \iid\ processes taking values in a countably infinite space.

The  assumption``taking more than two values'' in \cref{thm:iid-isomorphic} seems artificial, but we currently do not know how to remove it.  It does not appear that Seward's methods from \cite{seward2018bernoulli}, which have been used to remove a similar restriction in  Bowen's result \cite{bowen2012almostornstien} in the free-transitive case, can be directly applied, even for specific cases such as the case where $\V=\mathbb{Z}^d$ and $\Gamma$ is the group of automorphisms of the standard Cayley graph of $\mathbb{Z}^d$ (this seems to be non-trivial even for $d=1$).

We mention that some  groups of permutations are not Ornstein. For instance, if $\Gamma$ consists of all the permutations of $\V$ (or the countable subgroup consisting of all permutations that fix all but finitely many elements of $\V$), then de Finetti's theorem implies that any ergodic $(\V,\Gamma)$-process is actually \iid\ and that a pair of countable-valued \iid\ $(\V,\Gamma)$-processes are isomorphic if and only if the corresponding distributions of the marginals are equal up to renaming the symbols.  More generally, if $g \in \Gamma \setminus \{1_\Gamma\}$  moves only finitely many elements of $\V$, then the probability of the event $\{ g(X)=X\}$ is a non-trivial invariant of isomorphism for \iid\ processes. We refer to \cref{rem:non_aperiodic_example} in \cref{sec:nice-groups} for a different  example of a  permutation group $\Gamma$ which admits non-isomorphic \iid\ processes of equal entropy.


\subsection{Organization of the paper}
In \cref{sec:defs_notation}, we introduce some new definitions and recall some standard terminology.  In \cref{sec:relative_fd_Z}, we generalize a theorem of Smorodinsky~\cite{smorodinsky1992finitary}, stating that any two finite-valued finitely dependent $\Z$-processes of equal entropy are finitarily isomorphic.  As in other results in ergodic theory, we do so by formulating a version that is ``relative to a factor''. Additionally, we also allow for countable-valued processes. This ``relative'' generalization  provides us with additional flexibility needed for our applications and, arguably,  streamlines the proof for  Smorodinsky's original result. Also in \cref{sec:relative_fd_Z}  we prove a  finitary isomorphism theorem for \iid\ processes taking values in Polish spaces. We further introduce the notion of ``(relatively) pro-dependent processes'', and show that it gives a characterization of processes that are (relatively) finitarily isomorphic to \iid. In \cref{sec:sequential}, we apply the results obtained in  \cref{sec:relative_fd_Z} to deduce a version of \cref{thm:main} for $\Z$-processes. \cref{sec:relative_fd_Z,sec:sequential} deal exclusively with $\Z$-processes.
	
	In \cref{sec:nice-groups}, we introduce and discuss the notions of ``semi-nice'' and ``nice'' permutation groups. We also introduce the definition of entropy for $(\V,\Gamma)$-processes, where $\Gamma$ is a nice amenable group, and prove that some basic facts about entropy for classical $\Gamma$-processes hold in this setting as well. In \cref{sec:z_type_order}, we discuss the existence of certain random total orders on $\V$ obtained as finitary factors of \iid\ $(\V,\Gamma)$-processes. This is used to reduce \cref{thm:iid-isomorphic} and \cref{thm:main} to the case of $\Z$-processes. The reduction of the latter is carried out in \cref{sec:proof_of_thm_main}, which completes the proof of \cref{thm:main}. We conclude with some further remarks and related open questions in \cref{sec:conclusion}.

\subsection{Acknowledgement}  We thank Yair Glasner for helpful discussions regarding amenable actions of locally compact groups. Parts of this work have been carried out in the University of British Columbia and the Pacific Institute for Mathematical Sciences. The authors are thankful for the warm hospitality. Research of TM was supported in part by the ISF grant 1052/18. Research of YS was supported in part by NSERC of Canada.

\section{Definitions and notation}\label{sec:defs_notation}

Let $\V$ be  a countable set, let $\Gamma$ be a group of permutations of $\V$, and let $\cA$ be a Polish space. Given  an $\cA$-valued function $X \in \cA^\V$ on $\V$, we write $X_v$ for the value of $X$ at $v \in \V$ and $X_F$ for the restriction of $X$ to $F \subset \V$. The group 
$\Gamma$ acts on $\cA^\V$. For concreteness, we use the left action given by $g(x)_{v} := x_{g^{-1}(v)}$, for $x \in \cA^{\V}$ and $g \in \Gamma$.
An $\cA$-valued $(\V,\Gamma)$-process $X= (X_v)_{v \in \V}$
is a random function from $\V$ to  $\cA$, whose distribution is invariant with respect to the action of $\Gamma$.  When $\V$ and $\Gamma$ are clear from the context, we say ``$X$ is process'', suppressing $\V$ and $\Gamma$ from the notation.
A \textbf{joining} of two processes $X$ and $Y$ is a $\Gamma$-invariant coupling of $X$ and $Y$. When we say that a process $X$ takes more than $k$ values, we mean that there is no set $\cA'$ with $|\cA'|=k$ such that $X_v \in \cA'$ almost surely.

Given two random variables $U$ and $V$ (on a common probability space), we denote by $\cL(U \mid V)$ the conditional distribution of $U$ given $V$. Observe that if $U$ takes values in $\cA$, then $\cL(U \mid V)$ is a random variable taking values in the space of probability measures on $\cA$.

\medskip\noindent
\textbf{Partial processes.}
Let  $\star \not \in \cA $ be a ``new symbol'', to be  interpreted as ``undefined''. For $I \subseteq \V$ and $x,x' \in (\cA \cup \{\star\})^I$, we say that $x'$ \textbf{extends} $x$ if $x'_v=x_v$ for every $v \in I$ such that $x_v \ne \star$. A \textbf{partial process} of an $\cA$-valued process $X$ is an $(\cA \cup \{\star\})$-valued process $\tilde X$ that comes with a joining $(X,\tilde X)$ so that $X$ almost surely extends $\tilde X$. We refer to $\Pr( \tilde X = \star)$ as the \textbf{uncertainty} of the partial process $\tilde X$. We use the notation $\tilde X \preceq X$ to indicate that $\tilde X$ is a partial process of $X$.

\medskip\noindent
\textbf{Factors.}
Let $\cX$ and $\cY$ be two Polish spaces on which $\Gamma$ acts (measurably).
Let $X \in \cX$ and $Y \in \cY$ be $\Gamma$-invariant random variables defined on a common probability space.
Let $\varphi \colon \cY \to \cX$ be measurable.
We say that $\varphi$ is a \textbf{factor map} from $Y$ to $X$ if it is $\Gamma$-equivariant, i.e., for every $\gamma \in \Gamma$, it almost surely holds that $\varphi(\gamma Y) = \gamma \varphi(Y)$, and $X = \varphi(Y)$ almost surely.
We say that \textbf{$X$ is a factor of $Y$} if there exists a factor map from $Y$ to $X$.

We will mostly be interested in the situation where $\cX=\cA^\V$ and $\cY=\cB^V$ for some Polish spaces $\cA$ and $\cB$, in which case $X$ is an $\cA$-valued $(\V,\Gamma)$-process and $Y$ is a $\cB$-valued $(\V,\Gamma)$-process.

\medskip\noindent
\textbf{Finitary factors.}
There are two main notions of finitary factors in the literature, which coincide for the class of discrete-valued processes. One definition is based on stopping times, while the other involves topology. To distinguish between the two notions, we call the first stop-finitary and the second topo-finitary. 

\smallskip
\noindent
\emph{Stopping-time definition.}
Let $\cA$ and $\cB$ be two Polish spaces.
Let $X$ and $Y$ be $\cA$- and $\cB$-valued $(\V,\Gamma)$-processes, defined on a common probability space.
Fix an enumeration $\{v_0,v_1,v_2,\dots\}$ of $\V$, and let $\cF=(\cF_n)_{n \ge 0}$ be the natural filtration associated to $(Y_{v_0},Y_{v_1},\dots)$.  Given a stopping time $\tau$ (with respect to the filtration $\cF$), let $\cF_\tau$ be the $\sigma$-algebra consisting of all events $F$ such that $F \cap \{ \tau \le n \} \in \cF_n$ for all $n$.
Let $\varphi \colon \cB^\V \to \cA^\V$ be a factor map from $Y$ to $X$.
We say that $\varphi$ is \textbf{stop-finitary} if $\varphi(Y)_{v_0}$ is $\cF_\tau$-measurable for some almost surely finite stopping time $\tau$. It is easy to see that the definition does not depend on the enumeration of $\V$. Note also that the finitaryness of $\varphi$ is not affected by the modification of $\varphi$ on a null set (with respect to $Y$).

\smallskip
\noindent
\emph{Topological definition.}
Let $\cX$ and $\cY$ be two Polish spaces.
Let $X \in \cX$ and $Y \in \cY$ be $\Gamma$-invariant random variables defined on a common probability space.
Let $\varphi \colon \cY \to \cX$ be a factor map from $Y$ to $X$.
We say that $\varphi$ is \textbf{topo-finitary} if it is continuous when restricted to a set of full measure, i.e., if there exists a
measurable set $\Omega \subset \cY$ such that $\varphi|_\Omega$ is continuous and $Y \in \Omega$ almost surely. 

\smallskip
\noindent
\emph{Examples.}
\begin{itemize}

\item Let $\cA := [0,1]$ and $\cB:=\{0,1\}$.
Let $Y$ be the \iid\ process consisting of fair coins flips and define $X$ by letting $X_n$ be the real number whose binary expansion is $(Y_n,Y_{n+1},\dots)$. Then $X$ is a topo-finitary factor of $Y$ (in fact, the associated map is continuous everywhere), but is not a stop-finitary factor.

\item A similar example to above in which $\cA$ is countable (but not discrete) is as follows. Let $\cA := [0,1] \cap \Q$ and $\cB:=\{0,1\}$. Let $Y$ be as before and let $X_n$ be the number whose binary expansion is $(\1_{\{Y_{n+j}=1\text{ for }i \le j \le 2i\}})_{i \ge 1}$. A simple application of Borel--Cantelli shows that $X_n \in \cA$ almost surely. Then $X$ is a topo-finitary factor of $Y$, but not a stop-finitary factor.

\item Let $\cA := \{0,1\}$ and $\cB:=[0,1]$.
Let $Y$ be the \iid\ process consisting of uniform variables on $[0,1]$.
Let $B \subset \cB$ be a Borel set and define $X$ by $X_n := \1_{\{Y_n \in B\}}$. Then $X$ is a stop-finitary factor of $Y$ (with stopping time $\tau=0$), and while many non-trivial choices of $B$ (e.g., $[0,\frac12]$) make $X$ a topo-finitary factor of $Y$, there are also many choices (e.g., any fat Cantor set) for which $X$ is not a topo-finitary factor of $Y$.
\end{itemize}
The examples demonstrate that, in general, the two notions of finitary factors are not comparable.
The following lemma shows that the two notions are in fact equivalent when $\cX=\cA^\V$ and $\cY=\cB^V$ and the state spaces $\cA$ and $\cB$ are discrete (note that a discrete Polish space is at most countable).

\begin{lemma}\label{lem:finitary-equiv}
Let $\cA$ and $\cB$ be Polish spaces.
Let $X$ and $Y$ be $\cA$- and $\cB$-valued $(\V,\Gamma)$-processes, defined on a common probability space. Let $\varphi \colon \cB^\V \to \cA^\V$ be a factor map from $Y$ to $X$.
\begin{itemize}
 \item If $\cA$ is discrete and $\varphi$ is topo-finitary, then $\varphi$ is stop-finitary.
 \item If $\cB$ is discrete and $\varphi$ is stop-finitary, then $\varphi$ is topo-finitary.
\end{itemize}
\end{lemma}

\begin{proof}
Suppose first that $\varphi$ is topo-finitary with $\cA$ discrete. Let us show that $\varphi$ is also stop-finitary.
Let $\Omega \subset \cB^\V$ be such that $\varphi|_\Omega$ is continuous and $Y \in \Omega$ almost surely.
By continuity, we can partition $\Omega$ into countably many relatively open sets $\{\Omega_a\}_{a \in \cA}$ defined by $\Omega_a := \{ y \in \Omega : \varphi(y)_{v_0}=a\}$.
For $y \in \cB^\V$ and $n \ge 0$, let $B_n(y)$ denote the set of all $y' \in \cB^\V$ which agree with $y$ on $\{v_0,\dots,v_n\}$.
Define $\tau := \min\{n \ge 0 : B_n(Y) \subset \Omega_a \text{ for some }a \in \cA \}$.
Using that $\cB^\V$ has the product topology and that each $\Omega_a$ is open in $\Omega$, it follows that $\tau$ is an almost surely finite stopping time  and that $\varphi(Y)_{v_0}$ is $\cF_\tau$-measurable.
This shows that $\varphi$ is stop-finitary.

Now suppose that $\varphi$ is stop-finitary with $\cB$ discrete. By definition of the product topology on $\cA^\V$ (and since $\varphi$ is equivariant), it suffices to show that the map $y \mapsto \varphi(y)_{v_0}$ from $\cB^\V$ to $\cA$ coincides with a continuous function $F$ on a set $\Omega \subset \cB^\V$ of full measure. Let $\tau$ be an almost surely finite stopping time for which $X_{v_0}$ is $\cF_\tau$-measurable. The random variable $Z:=(Y_{v_0},\dots,Y_{v_\tau})$ almost surely takes values in the discrete countable space of finite words over $\cB$, and the $\sigma$-algebra generated by it is $\cF_\tau$. Thus, there exists a function $f \colon \cS \to \cA$ such $X_{v_0}=f(Z)$ almost surely, where $\cS$ is the support of $Z$. For $s\in \cS$, write $[s]$ for the set of $y \in \cB^\V$ having prefix $s$. Let $\Omega := \bigcup_{s \in \cS} [s]$ and define $F(y) := f(s)$ for $y \in [s]$. This is well defined since $\cS$ is prefix free (no word in $\cS$ is the prefix of another word in~$\cS$). Then $F$ is continuous and $\Pr(Y \in \Omega)=1$.
\end{proof}

When $\cA$ and $\cB$ are discrete, the proof of \cref{lem:finitary-equiv} shows that the notion of topo-finitary factor does not change (up to a null set) if in its definition we require $\varphi$ to be continuous on $\Omega$ as opposed to the weaker property that its restriction to $\Omega$ is continuous. In other words, if $X=\varphi(Y)$ almost surely for some equivariant function $\varphi$ which is continuous when restricted to a set of full measure, then $X=\tilde\varphi(Y)$ almost surely for some equivariant function $\tilde\varphi$ which is continuous on a set of full measure.
In fact, this is true in much larger generality (see, e.g., \cite{costantini2000extensions}), though none of this will be important for the results in this paper.

When $\cA$ and $\cB$ are discrete, another interpretation of the two equivalent definitions that $X$ is a finitary factor of $Y$ is that there almost surely exists a finite set $V \subset \V$ such that $(Y_v)_{v \in V}$ determines $X_{v_0}$ (in the sense that $X_{v_0}$ is almost surely constant given the witnessed values).
We leave the verification of this as an exercise for the reader.


In the statement of \cref{thm:amenable_groups_efficient_finitary_factors_of_iid,thm:main}, the assumption is that $X$ is a finitary factor of an \iid\ processes, which could  take  values in an uncountable non-discrete Polish space.  In this case ``finitary'' should be understood as stop-finitary.  Since $X$ takes values in a countable set, taking the discrete topology,  ``stop-finitary'' is a weaker assumption than ``topo-finitary''  which leads to  stronger theorems.

\medskip
\noindent
\textbf{Block factors.}	
 A $\Z$-process $X$ is an \textbf{$m$-block factor} of a process $Y$ if $X_0$ is measurable with respect to $Y_{[-m,m)}$.
We say that $X$ is a block factor of $Y$ if it is an $m$-block factor for some finite $m$. Clearly, if $X$ is a block factor of $Y$, then it is a stop-finitary factor of $Y$.
 If $X$ is $\cA$-valued and $Y$ is $\cB$-valued, with $\cA$ and $\cB$ finite, then $X$ is a block factor of $Y$ if and only if there exists a continuous equivariant map $\pi:\cB^\mathbb{Z} \to \cA^\mathbb{Z}$ such that $X=\pi(Y)$ almost surely. In particular, in this case, a block factor is also a topo-finitary factor. More generally, by \cref{lem:finitary-equiv}, this is the case whenever $\cB$ is discrete (but need not be the case in general; recall the examples above).

\medskip
\noindent
\textbf{Finitary (relative) isomorphism.}
Let $X$ and $Y$ be two $(\V,\Gamma)$-processes. We say that $X$ and $Y$ are \textbf{finitarily isomorphic} if there is an isomorphism between them which is finitary and whose inverse is also finitary.
Now let $X,Y,W$ be $(\V,\Gamma)$-processes, with joinings $(X,W)$ and $(Y,W)$ given implicitly in the background. We say that
\textbf{$X$ and $Y$ are finitarily isomorphic relative to $W$} if there is a finitary isomorphism $\pi$ between $(X,W)$ and $(Y,W)$ which fixes the $W$-component in the sense that $\pi(X,W)=(Y,W)$ almost surely.

\medskip
\noindent
\textbf{Entropy.}
The Shannon entropy of a random variable $X$ taking values in a countable set $\cA$  is 
\[ H(X)= -\sum_{a \in \cA} \Pr( X = a) \log \Pr(X = a).\]	
The conditional Shannon entropy of $X$ given a random variable $W$ (defined on  common probability space) is 
\[ H(X\mid W)= -\E \sum_{a \in \cA} \Pr( X = a\mid W) \log \Pr(X = a \mid W).\]	
Now suppose that $X= (X_n)_{n \in \Z}$ is an $\cA$-valued $\Z$-process with $H(X_0) <\infty$.
The Kolmogorov--Sinai entropy of $X$ is given by
\[h(X)= \lim_{n \to \infty}\frac{1}{n} H(X_{[1,n]}).\]
In particular, if $X$ is an \iid\ process, then $h(X) = H(X_0)$.
In general, when $H(X_0) = \infty$ the formula for $h(X)$ above is incorrect, and the Kolmogorov--Sinai entropy of $X$ can be finite.
However, if $X$ is an \iid\ process and $H(X_0)=\infty$, then $h(X)=\infty$.
For a process $X$ for which $H(X_0)$ is not necessarily finite (and possibly $X_0$ takes an uncountable set of values), $h(X)$ can be defined as the supremum of $h(X')$ over all processes $X'$ which are factors of $X$ and satisfy $H(X'_0) < \infty$.
A fundamental property of Kolmogorov--Sinai entropy is that it is monotone under factors: If $X$ is a factor of $Y$ then $h(X) \le h(Y)$.
The above extends in a natural way to condition entropy: when $X$ and $W$ are two $\Z$-processes with a common joining and $H(X_0 \mid W) < \infty$, the conditional Kolmogorov--Sinai entropy of $X$ given $W$ is given by
\[h(X \mid W)= \lim_{n \to \infty}\frac{1}{n} H(X_{[1,n]}\mid W).\]

The Kolmogorov--Sinai entropy of a $\Gamma$-process $X= (X_v)_{v \in \Gamma}$ with $H(X_0) <\infty$ and $\Gamma$ a discrete countable amenable group is given by
\[ h(X) = \inf_{\substack{F \subset \Gamma \\ 0<|F|<\infty}} \frac{H(X_F)}{|F|} .\]
In \cref{sec:entropy}, we extend this definition to our setting of $(\V,\Gamma)$-processes over nice amenable permutation groups.

\section{A relative finitary isomorphism theorem for finitely dependent $\Z$-processes}\label{sec:relative_fd_Z}

In this section we exclusively deal with  $\Z$-processes, namely  bi-infinite sequences of random variables  with a shift-invariant distribution. 

For an integer $k \ge 0$, we say that a $\Z$-process $X$ is \textbf{$k$-dependent} if $X_A$ and $X_B$ are independent for any $A,B \subset \Z$ such that $|a-b| > k$ for all $a\in A$ and $b\in B$. A process is \textbf{finitely dependent} if it is $k$-dependent for some $k \ge 0$.
Smorodinsky~\cite{smorodinsky1992finitary} proved that any two finite-valued finitely dependent $\Z$-processes of equal entropy are finitarily isomorphic.
We will show that this result also holds for countable-valued processes (with finite or infinite entropy):  
	\begin{thm}\label{thm:smorodinsky_k_dependent}
		Any two equal-entropy finitely dependent $\Z$-processes taking at most countably many values are finitarily isomorphic.
	\end{thm}

	In the statement of \cref{thm:smorodinsky_k_dependent}, the finite or countable sets in which the process takes values are assumed to be discrete, so that the notions of topo-finitary and stop-finitary coincide.

We now introduce further definitions needed to formulate a ``relative'' version of \cref{thm:smorodinsky_k_dependent}. 
Let $(X,W)$ be a joining of two $\Z$-processes $X$ and $W$.
For an integer $k \ge 0$, we say that $X$ is \textbf{$k$-dependent relatively to $W$} if, almost surely, $X_A$ and $X_B$ are conditionally independent given $W$ for any $A,B \subset \Z$ such that $|a-b| > k$ for all $a\in A$ and $b\in B$.
More generally, given an integer-valued process $K=(K_n)_{n \in \Z}$ which is a factor of $W$, we say that $X$ is \textbf{$K$-dependent relatively to $W$} if, almost surely, $X_A$ and $X_B$ are conditionally independent given $W$ for any $A,B \subset \Z$ (measurable with respect to $W$) such that $|a-b|>\max\{K_a,K_b\}$ for all $a \in A$ and $b \in B$.
We say that $X$ is \textbf{finitarily $K$-dependent relatively to $W$} if it is $K$-dependent relatively to $W$, 
and $(\cL(X_{[n-m,n+m]} \mid W))_{n \in \Z}$ is a stop-finitary factor of $W$ for any integer $m \ge 0$.
We say that $X$ is \textbf{finitarily dependent relatively to $W$} if it is finitarily $K$-dependent for some process $K$ which is a stop-finitary factor of $W$.
Recall that $W$ is \textbf{aperiodic} if the probability that there exists an integer $p \ge 1$ such that $W_{n+p} = W_n$ for all $n \in \Z$ is zero (equivalently, the $\Z$-action associated with $W$ is essentially free).


\begin{thm}\label{thm:relative_smorodinsky_k_dependent}
	Let $W$ be an aperiodic ergodic $\Z$-process
and let $X$ and $\tilde X$ be two countable-valued $\Z$-processes, both finitarily dependent  relative to $W$, such that $h(X \mid W)=h(\tilde X \mid W)$.
	Then $X$ and $\tilde X$ are finitarily isomorphic relative to $W$.
\end{thm}

	\begin{remark}
		In the statement of \cref{thm:relative_smorodinsky_k_dependent} above all instances of the notion ``finitary'' are ``stop-finitary''. In particular, our assumption is that  $(\cL(X_{A+n} \mid W))_{n \in \Z}$ is a stop-finitary factor of $W$ for any finite $A \subset \Z$, and similarly for $\tilde X$ (see the next remark on why such an assumption is needed). Because $(\cL(X_{A+n} \mid W))_{n \in \Z}$ takes values in a non-discrete topological space even when $X$ and $W$ themselves do, there is a genuine  distinction between the notions of topo-finitary and stop-finitary here. The conclusion states that the processes are stop-finitarily isomorphic relative to $W$. We expect that an analogous statement should hold if we switch to ``topo-finitary'' both in the assumptions and in the conclusion, but we do not pursue this here.
	\end{remark}

	\begin{remark}\label{rem:finitary-cond-law}
	Let us explain why one cannot drop the assumption that $(\cL(X_{A+n} \mid W))_{n \in \Z}$ is a finitary factor of $W$. Consider a process $W$ and a factor $X$ of it, which is not a finitary factor of it. Then $X$ is 0-dependent relative to $W$ (since it is deterministic given $W$), but it is not finitarily isomorphic to an \iid\ process relative to $W$ (since this would mean that the \iid\ process is a constant process and hence that $X$ is a finitary factor of $W$).
One can also construct examples where $X$ consists of conditionally independent non-constant random variables given $W$.
\end{remark}

In the statement of \cref{thm:relative_smorodinsky_k_dependent}, if $W$ is  trivial process (in which case the process $K$ must be a deterministic constant), the statement becomes \cref{thm:smorodinsky_k_dependent}. In fact, the requirement that $W$ be aperiodic can be removed from the statement of \cref{thm:relative_smorodinsky_k_dependent}: The only remaining case is that $W$ is a non-trivial periodic ergodic process meaning there exists $p>1$ such that $W_{n+p}=W_n$ almost surely. This case can be dealt with by essentially following the same steps described in \cref{sec:thm_smorodinsky}.
Our proof of  \cref{thm:smorodinsky_k_dependent}, whose essence is a reduction to \cref{thm:relative_smorodinsky_k_dependent}, does not rely on Smorodinsky's paper \cite{smorodinsky1992finitary}, except for one specific claim (\cref{lem:twins}). Somewhat surprisingly, our proof of \cref{thm:relative_smorodinsky_k_dependent} actually avoids certain extra complications confronted when $W$ is a trivial or periodic process (but presents its own unique challenges).


As a corollary of \cref{thm:relative_smorodinsky_k_dependent}, we obtain the following result.

%

We say that a process $X$ is \textbf{finitarily pro-dependent} relative to $W$ if there exists a sequence $(X^{(n)})_{n=1}^\infty$ of partial processes increasing to $X$ such that for each $n$:
\begin{itemize}
 \item $X^{(n)}$ is a finitary factor of $(W,X)$.
 \item $X^{(n)}$ is finitarily dependent relative to $(W,X^{(1)},\dots,X^{(n-1)})$.
\end{itemize}

Clearly, $X$ being finitarily dependent relative to $W$ implies that $X$ is finitarily pro-dependent relative to $W$.

\begin{thm}\label{thm:pro-dependent}
	Let $W$ be an aperiodic ergodic $\Z$-process
and let $X$ and $\tilde X$ be two countable-valued $\Z$-processes, both of which are finitarily pro-dependent relative to $W$, such that $h(X \mid W)=h(\tilde X \mid W)$.
	Then $X$ and $\tilde X$ are finitarily isomorphic relative to $W$.
\end{thm}

\begin{cor}\label{cor:finitarily_factor_relatively_finitarily_dependent}
	Let $X$ be  a countable-valued ergodic process. Suppose that $X$ has a finitary factor $W$ relatively to which it is finitarily pro-dependent. Then  $X$ is finitarily isomorphic to \iid$\times W$.
\end{cor}

\subsection{Proof of \cref{thm:smorodinsky_k_dependent}, assuming \cref{thm:relative_smorodinsky_k_dependent}\label{sec:thm_smorodinsky}}

The first ingredient in the proof is a so-called marker process. A \textbf{marker process} is any non-trivial $\{0,1\}$-valued process. We will be interested in marker processes which arise as finitary factors of a given process $Z$. We say that a marker process is a \textbf{marker process for $Z$} if it is a finitary factor of $Z$. A typical way to construct a marker process for $Z$ is to look at the locations of occurrences of some fixed pattern. That is, given a pattern $u=(u_0,\dots,u_m)$, we consider the marker process $M$ defined by
\[ M_i = \1_{\{ Z_i=u_0,\, Z_{i+1}=u_1,\dots,\,Z_{i+m}=u_m \}} .\]
We call any such marker process an \textbf{occurrence marker process}.

Our first goal will be to construct a marker process $M$ for a given finitely dependent process $X$ in such a way that makes $X$ finitarily dependent relatively to $M$. To illuminate the potential difficulty in doing so, suppose that $X$ is $k$-dependent and consider any occurrence marker process $M$ for $X$. Then $X_{(-\infty,-k)}$ and $X_{[0,\infty)}$ are conditionally independent given that $M_0=1$.
On the other hand, this is no longer necessarily true when also conditioning on the absence or presence of other markers in the vicinity: given that $M_0=1$ and given $M_{(-\infty,0)}$, we cannot in general say that $X_{(-\infty,-k)}$ and $X_{[0,\infty)}$ are conditionally independent (even if the pattern used for the occurrence marker process has no self overlaps).
Similarly, if $i \in \Z$ is a random integer which depends on $M$ and satisfies $M_i=1$ almost surely, then it is not necessarily the case that $X_{(-\infty,i-k)}$ and $X_{[i,\infty)}$ are conditionally independent given $M$.
Let us further illustrate the problem by an example: let $Y$ be any non-trivial $\{0,1\}$-valued \iid\ process and let $X$ be defined by $X_i=2$ if $Y_i = Y_{i+1}$ and $X_i=Y_i$ otherwise. Note that $X$ is a block factor of $Y$ (in particular, $X$ is finitely dependent) and the factor map is invertible (and the inverse is finitary). Let $M$ be the occurrence marker process for $X$ given by the locations of 2s, i.e., $M_i := \1_{\{X_i=2\}}$. It follows that $X$ is a 2-to-1 extension of $M$, and thus not a finitary factor of \iid\ relative to $M$ and also not finitarily dependent relative to $M$. The obstruction to the latter can be seen as two-fold: there is no process $K$ for which $X$ is $K$-dependent relative to $M$, but also, unless $Y$ happens to consist of \emph{unbiased} bits, the conditional law of $X_0$ given $M$ is not a finitary function of $M$ (in fact, it is not even a topo-finitary function of $M$), so that even if such $K$ existed, $X$ would not be \emph{finitarily} $K$-dependent relative to $M$ (recall \cref{rem:finitary-cond-law}). Moreover, $X$ and $M$ are two finitely dependent processes with equal entropy and having states with a common distribution, but there is no isomorphism of the two which maps one state to the other. 
These somewhat subtle issues lead us to the next definition.

Let $X$ be a process. We say that a marker process $M$ for $X$ is a \textbf{good marker process} for $X$ if there exists an integer $m \ge 1$ such that on the event that $M_0=1$, almost surely, $X_{(-\infty,-m]}$ and $X_{[m,\infty)}$ are conditionally independent given $M$ and their conditional distributions depend only on $M_{(-\infty,0]}$ and $M_{[0,\infty)}$, respectively.
Equivalently, if $(i_j)_{j=-\infty}^\infty$ is a random sequence of integers which is measurable with respect to $M$ and almost surely satisfies $i_{j+1}>i_j+2m$ and $M_{i_j}=1$ for all $j$, then, given $M$, almost surely, $\{ X_{[i_j+m,i_{j+1}-m]}\}_j$ are conditionally independent and the conditional distribution of each $X_{[i_j+m,i_{j+1}-m]}$ depends only on $M_{[i_j,i_{j+1}]}$.


\begin{lemma}\label{lem:good-markers}
Let $X$ be a process and let $M$ be a good marker process for $X$. Then $X$ is finitarily dependent relatively to $M$.
\end{lemma}
\begin{proof}
Let $m$ be the integer guaranteed by the definition of a good marker process.
Define a process $K$ by  $K_i := m+\max\{\ell^+_i,\ell^-_i\}$, where $\ell^\pm_i = \min\{ \ell>m : M_{i \pm \ell}=1 \}$.
Clearly, $K$ is a finitary factor of $M$, and hence also of $X$.
It is straightforward to check that $X$ is finitarily $K$-dependent relatively to $M$.
\end{proof}

Suppose now that $X$ is finitely dependent.
While we have seen that an occurrence marker process for $X$ need not be a good marker process for $X$, as we now show, such a marker process always contains within it a good marker process. A marker process $M'$ is a \textbf{finitary dilution} of a marker process $M$ if it is a finitary factor of it and $M'_n \le M_n$ for all $n$. We will also need to know that the dilution procedure does not depend on $X$, but only on the marker process itself. To state this precisely, it is convenient to allow any marker process which is a block factor of $X$, rather than only occurrence marker processes for $X$.

\begin{lemma}[good marker process]\label{lem:diluted-marker-process-universal}
Let $M$ be a finitely dependent maker process. Then for every $k,\ell \in \N$ there exists a marker process $M'$ which is a finitary dilution of $M$ such that if $X$ is $k$-dependent and $M$ is an $\ell$-block factor of $X$, then $M'$ is a good marker process for $X$.   
\end{lemma}

\begin{proof}
Fix $M,k,\ell$.
We first construct a dilution $M''$ of $M$ as a block factor of $M$, and then we construct $M'$ as a finitary dilution of $M''$. We will then show that $M'$ satisfied the claimed property.

Since $M$ is finitely dependent, there exist $C,c>0$ such that $\Pr(M_{[0,m)}=u)<Ce^{-cm}$ for any $m \in \N$ and any $u \in \{0,1\}^{m}$. Choose $m$ large enough so that $Ce^{-cm} < \frac{1}{k+\ell+m}$ and choose any $u \in \{1\} \times \{0,1\}^{m-1}$ for which $\Pr(M_{[0,m)}=u)>0$. Let $M''$ be the occurrence marker process for $M$ given by the pattern $u$. Then $M''$ is an $m$-block factor of $M$ and a finitary dilution of $M$, and it satisfies that $\Pr(M''_0=1)<\frac{1}{k+\ell+m}$.
In particular, gaps of size at least $g:=k+\ell+m$ between consecutive 1s in $M''$ occur with positive probability, and by ergodicity, also infinitely often in the past almost surely. We now explain the relevance of this.

Let $\Omega \subset \{0,1\}^\Z$ be the collection of bi-infinite sequences $(a_n)_{n \in \Z} \in \{0,1\}^\Z$ with the property that
\[
\inf \left\{ n \in \Z~:~a_{[n-g,n]} = 0^g1
\right\} = -\infty.
\]
For a sequence $a \in \Omega$, we define a sequence $\overline{a} \in \Omega$ by choosing a subset of 1s in $a$ as follows: we first take those 1s which are preceded by $g$ zeros, then we force the $g$ symbols succeeding these 1s to be zeros, and we repeat indefinitely. Formally, we first define $\varphi \colon \Omega \to \Omega$ by
\[ \varphi(a)_n := \begin{cases}
0 & \text{if } 1 \le \ell_n \le g \\
a_n & \text{otherwise}
\end{cases},\quad\text{where } \ell_n := \min \{ \ell \ge 0 : a_{[n-\ell-g,n-\ell]} = 0^g1 \}, \]
and then we define
\[ \overline{a} := \lim_{i \to \infty} \varphi^i(a) .\]
Note that $\varphi(a)$ is a dilution of $a$ and that $\varphi$ is equivariant. In particular, $\overline{a}$ is well defined and is a dilution of $a$, and the map $a \mapsto \overline{a}$ is equivariant. Observe also that every 1 in $\overline{a}$ is preceded/succeeded by $g$ zeros, that $\varphi(\overline{a})=\overline{a}$ and that any other $b$ such that $\overline{a} \le b \le a$ has $\overline{b}=\overline{a}$ and $\varphi(b) \neq b$.
Moreover, it is straightforward to check that if $a\in\Omega$ and $n \in \Z$, then
\begin{enumerate}
 \item The value of $\bar{a}_n$ depends only on a finite past of $a$ up to $n$. More precisely, if $a_{[m,m+g]}=0^g1$ for some $m \le n-g$ then  $\bar{a}_n=\bar{b}_n$ for any $b \in \Omega$ which agrees with $a$ on $[m,n]$.
 \item If $\bar{a}_n=1$ then the future of $\bar{a}$ after $n$ depends only on the future of $a$ after $n+g$. More precisely, if $\bar{a}_n=1$ then $\bar{a}$ and $\bar{b}$ agree on $[n,\infty)$ for any $b \in \Omega$ which has $\bar{b}_n=1$ and agrees with $a$ on $[n+g,\infty)$.
\end{enumerate}

Let us come back to the marker process. As $M'' \in \Omega$ almost surely, the diluted marker process $M' := \overline{M''}$ is defined almost surely as a factor of $M''$. The first property above shows that $M'$ is a finitary dilution of $M''$ and thus also of $M$.
Toward showing that $M'$ satisfies the claimed property, let $X$ be a $k$-dependent process such that $M$ is an $\ell$-block factor of $X$. Since $M''$ is an $(\ell+m)$-block factor of $X$, the first and second properties together imply that $(M'_{(-\infty,0]},X_{(-\infty,-g]})$ and $(M'_{[0,\infty)},X_{[g,\infty)})$ are conditionally independent given that $M'_0=1$, since the former depends only on $X_{(-\infty,\ell+m)}$, as does the event we are conditioning on, and the latter depends only on $X_{[g,\infty)}$ given the conditioning. It follows from this that $M'$ is a good marker process for $X$.
\end{proof}

We are almost ready to describe the reduction of \cref{thm:smorodinsky_k_dependent} to \cref{thm:relative_smorodinsky_k_dependent}. Before doing so, we need one additional claim, which is essentially taken from~\cite{smorodinsky1992finitary}. Say that two finitely dependent processes are \textbf{twins} if they have the same entropy and there exist two occurrence marker processes, one for each process, which have the same distribution.

\begin{lemma}[\cite{smorodinsky1992finitary}]\label{lem:twins}
Let $X$ and $\tilde{X}$ be two countable-valued finitely dependent processes of equal entropy. Then there exist four finitely dependent processes $X^1,X^2,X^3,X^4$ such that every two consecutive processes in $(X,X^1,X^2,X^3,X^4,\tilde{X})$ are twins.
\end{lemma}

\begin{proof}
In~\cite[Section~3]{smorodinsky1992finitary}, it is shown that, given a finitely dependent process $X$, there exists a finitely dependent process $X^1$ and an \iid\ process $X^2$ such that $X$ and $X^1$ are twins and $X^1$ and $X^2$ are twins. Applying this for $\tilde{X}$, we also get an \iid\ process $X^3$ and a finitely dependent process $X^4$ such that $\tilde{X}$ and $X^4$ are twins and $X^3$ and $X^4$ are twins. It remains to explain that the equal-entropy \iid\ processes $X^2$ and $X^3$ are twins. That this is indeed the case was shown in~\cite[Section~2]{keane1979bernoulli}.\footnote{The papers~\cite{keane1979bernoulli,smorodinsky1992finitary} deal only with finite-valued processes, but the arguments in \cite[Section~2]{keane1979bernoulli} and \cite[Section~3]{smorodinsky1992finitary} apply also to countable-valued processes with finite or infinite entropy.}
\end{proof}

\begin{proof}[Proof of \cref{thm:smorodinsky_k_dependent}]
Suppose that $X$ and $\tilde{X}$ are two finitely dependent processes of equal entropy taking at most countable many values. By the previous lemma, it suffices to show that they are finitarily isomorphic under the additional assumption that they are twins. We may thus assume this so that $X$ and $\tilde X$ have occurrence marker processes with the same distribution. Equivalently, there exists a marker process $M$ which (under some joining) is an occurrence marker process for $X$ and for $\tilde{X}$. By \cref{lem:diluted-marker-process-universal}, there exists a marker process $M'$ (obtained as a finitary dilution of $M$) which is a good marker process for $X$ and for $\tilde X$. \cref{lem:good-markers} tells us that $X$ and $\tilde X$ are each finitarily dependent relatively to $M'$. Finally, $h(X \mid M')=h(X)-h(M')=h(\tilde X)-h(M')=h(\tilde{X} \mid M')$. Thus, we have arrived at the situation of \cref{thm:relative_smorodinsky_k_dependent} with a process $W=M'$ which is ergodic and aperiodic (since it is a factor of a finitely dependent process). The conclusion of the theorem in this case is that $X$ and $\tilde{X}$ are finitarily isomorphic relatively to $M'$. However, since $M'$ is a finitary factor of each of $X$ and $\tilde X$, this means that $X$ and $\tilde{X}$ are finitarily isomorphic.
\end{proof}

\subsection{The Keane--Smorodinsky marriage lemma}
\label{sec:marriage}
We now present a certain formulation of the Keane--Smorodinsky marriage lemma \cite{keane1977class}. We give a self-contained proof here, closely following~\cite[Section $\S 3$]{keane1977class}.

Let $U$ and $V$ be two finite sets and 
let $\lambda$ be a probability measure on $U\times V$. We say that $v \in V$ is $\lambda$-\textbf{committed} if there is at most one $u \in U$ such that $\lambda(\{(u,v)\})>0$.
Since $U$ and $V$ are finite, a measure $\tilde\lambda$ on $U \times V$ is absolutely continuous with respect to $\lambda$ if and only if $\tilde \lambda(\{ (u,v)\}) =0$ whenever $\lambda( \{(u,v)\}) =0$ for $(u,v)\in U \times V$.


\begin{lemma}\label{lem:coupling_reduction}
Let $(U\times V,\lambda)$ be coupling of $(U,\rho)$ and $(V,\sigma)$ with $U$ and $V$ finite.
Then there exists another coupling $(U \times V,\tilde \lambda)$, which is absolutely continuous with respect to $\lambda$, such that
\[\left|\left\{v \in V : v \mbox{ is not } \tilde\lambda \mbox{-committed} \right\} \right| \le |U|-1.\]
\end{lemma}
\begin{proof}
	Let ${U \choose 2}$ denote the collection of unordered pairs in $U$ (subsets of $U$ having cardinality $2$). 
	Define 
	\[A(\lambda) := \left\{(\{u_1,u_2\},v)\in \tbinom U2\times V :  \lambda(\{(u_1,v)\}),\,\lambda(\{(u_2,v)\})>0 \right\}.\]
	Clearly $|A(\lambda)|$ is an upper bound on the number of $v \in V$ which are not $\lambda$-committed. To prove the lemma, we show that if $|A(\lambda)| \ge |U|$ then there exist a coupling $\tilde \lambda$ which is absolutely continuous with respect to $\lambda$ such that $|A(\tilde \lambda)|< |A(\lambda)|$. To this end, consider the multi-graph $G(\lambda)$ on the vertex set $U$ in which each $(\{u_1,u_2\},v) \in A(\lambda)$ represents an edge between $u_1$ and $u_2$. A simple cycle in $G(\lambda)$ is a sequence $C = \big((\{u_1,u_2\},v_1),(\{u_2,u_3\},v_2)\ldots,(\{u_n,u_1\},v_n)\big)$ with $e_i:=(\{u_i,u_{i+1}\},v_i) \in A(\lambda)$ (where we set $u_{n+1}:=u_1$ for notational ease) and $e_1,\ldots,e_n$ all distinct (a cycle of length $2$ is a pair of parallel edges).
	Since $|A(\lambda)| \ge |U|$, there is such a simple cycle in $G(\lambda)$.
	We assume without loss of generality that
	\[\lambda(\{(u_1,v_1)\}) = \min\big\{ \lambda(\{(u_i,v_i)\}),\,\lambda(\{(u_{i+1},v_i)\}) : 1 \le i \le n\big\}.\]
	Define $\tilde \lambda$ as follows:
	\begin{align*}
	\tilde \lambda (\{(u_i,v_i)\}) &:=\lambda (\{(u_i,v_i)\})-\lambda(\{(u_1,v_1)\}), &1 \le i \le n,\\
	\tilde \lambda (\{(u_{i+1},v_i)\}) &:= \lambda (\{(u_{i+1},v_i)\})+\lambda(\{(u_1,v_1)\}), &1 \le i \le n,
	\end{align*}
	and	 $\tilde \lambda (\{(u,v)\}):= \lambda (\{(u,v)\})$ in all other cases. It follows by direct verification that $\tilde \lambda$ is a coupling which is absolutely continuous with respect to $\lambda$ and that $|A(\tilde \lambda)|< |A(\lambda)|$. 
\end{proof}


\subsection{Proof of \cref{thm:relative_smorodinsky_k_dependent}}\label{sec:aperiodic-case}

Recall that a marker process $M$ is any non-trivial $\{0,1\}$-valued process, and that deleting some markers in a finitary manner produces a finitary dilution.
The occurrences of $M$ induce a random partition of $\Z$ into intervals: for $n \in \Z$, denote by $I^M_n$ the random interval containing $n$ in this partition (for concreteness, we include each occurrence of $M$ in the interval to its right).
The \textbf{minimal gap length} of $M$ is the largest $m \ge 1$ such that $|I^M_0| \ge m$ almost surely. Note that the density of a marker process with minimal gap length $m$ is at most $1/m$.

\begin{lemma}\label{lem:low-density-markers}
Any ergodic aperiodic marker process contains a finitarily diluted marker process with arbitrarily large minimal gap length.
\end{lemma}
\begin{proof}
Let $M$ be an ergodic aperiodic marker process. Let $m$ be the minimal gap length of $M$. It suffices to show that $M$ contains a finitarily diluted marker process with minimal gap length at least $m+1$. Since $M$ is aperiodic, the pattern $u=10^m$ consisting of a one followed by $m$ zeros occurs in $M$ with positive probability.
The occurrence marker process for $M$ corresponding to $u$ is a finitary dilution of $M$ with minimal gap length at least $m+1$.
\end{proof}

We will also require the following simple lemma about information of random variables.
Let $X$ be a random variable taking values in a countable set, and let $W$ be another random variable defined on the same probability space. Let $p_X$ denote the distribution function of $X$, i.e., $p_X(a) :=\Pr(X=a)$, and let $p_{X \mid W}(a) :=\Pr(X=a \mid W)$. Denote
\[\cI(X) = -\sum_{a \in \cA}\log p_X(a)\1_{[X=a]} \qquad\text{and}\qquad \cI(X \mid W) = -\sum_{a \in \cA}\log p_{X \mid W}(a)\1_{[X=a]} .\]

\begin{lemma}\label{lem:info-comparison-bound}
Let $X$ and $Y$ be discrete random variables. Then for any $t,s>0$,
\[ \Pr(\cI(Y) \le t) \le \Pr(X \neq Y) + \Pr(\cI(X) \le t+s) + e^{-s} .\]
\end{lemma}
\begin{proof}
We have
\[ \Pr(\cI(Y) \le t) \le \Pr(X \neq Y) + \Pr(-\log p_Y(X) \le t) \]
and
\[ \Pr(-\log p_Y(X) \le t) \le \Pr(-\log p_X(X) \le t+s) + \Pr(-\log p_X(X) > t+s, -\log p_Y(X) \le t) .\]
Since there are at most $e^t$ possible values $x$ having $-\log p_Y(x) \le t$, a union bound shows that the last term above is at most $e^{-(t+s)}e^t=e^{-s}$.
\end{proof}

Let $W$ be an ergodic aperiodic process and suppose that $M$ is a marker process which is a factor of $W$.
We say that a process $X$ is \textbf{$M$-block-dependent relative to $W$} if $X_{I^M_0}$ and $X_{\Z \setminus I^M_0}$ are conditionally independent given $W$. Equivalently,  $X_{I^{M}_{n_1}},\ldots,X_{I^{M}_{n_k}}$ are jointly conditionally independent given $W$, whenever $n_1,\ldots,n_k$ are random integers, measurable with respect to $W$, such that $I^M_{n_1},\dots, I^M_{n_k}$ are almost surely all distinct.
We say that $X$ is \textbf{finitarily $M$-block-dependent relative to $W$} if it is $M$-block-dependent relative to $W$ and $(\cL(X_{I^M_n} \mid W))_{n \in \Z}$ is a (stop-)finitary factor of $W$.

The relevance of finitary $M$-block-dependence will be clear in the proof of \cref{thm:relative_smorodinsky_k_dependent}.
On the one hand, it is a strengthening of finitary dependence: if $X$ is finitarily $M$-block-dependent relative to $W$, then it is also finitarily $K$-dependent for some process $K$ (with $K$ being a finitary factor of $W$ when $M$ is), but the converse is not necessarily true. On the other hand, it turns out that any process which is finitarily dependent relative to an aperiodic $W$ can be ``well approximated'' in a certain sense by a partial process which is finitarily $M$-block-dependent for some marker process $M$ which is a finitary factor of $W$.

The following is a consequence of the relative Shannon--McMillan--Breiman theorem when applied to $M$-block-dependent processes.

\begin{lemma}\label{lem:Shannon-McMillan-bidirectional}
	Let $V$ be a process having a marker process $M$  as factor, and let $(X,Y)$ be a process which is $M$-block-dependent relative to $V$ and satisfying that $h(X,Y \mid V)<\infty$. 
	Then, almost surely,
	\[  \cI(X_{[a,c]},Y_{[c,b]}\mid V) = (c-a)h(X \mid V) + (b-c)h(Y \mid V) + o(b-a)\]
	as $b-a \to \infty$, uniformly over $a,b,c$ satisfying that $a<0<b$ and $a<c<b$.

\end{lemma}

It seems plausible that the  conclusion of \cref{lem:Shannon-McMillan-bidirectional} holds even without the $M$-block-dependence assumption, but we do not currently know if this is indeed the case.

\begin{proof}
	Note first that because $(X,Y)$ is $M$-block-dependent relative to $V$, the assumption that $h(X,Y \mid V) < \infty$ implies that 
	$H(X_0,Y_0 \mid V) < \infty$ and also that $h(X \mid V) ,h(Y \mid V) < \infty$.
	
	We now show that, almost surely,
	\begin{equation}\label{eq:SMB_X_Y_sup_c}
		\max_{c \in [a,b]}\left| \cI(X_{[a,c]},Y_{[c,b]} \mid V) - \left(\cI(X_{[a,c]} \mid V)+\cI(Y_{[c,b]} \mid V)\right)\right| = o(b-a),
	\end{equation}
	as $b-a \to \infty$ with $a<0<b$.
	Indeed, if we denote $I^M_c = [c^-,c^+)$, i.e., $c^+$ is the first occurrence of $M$ in $(c,\infty)$ and $c^-$ is the last occurrence of $M$ in $(-\infty,c]$, then the fact that $(X,Y)$ is $M$-block-dependent relative to $V$ implies that for $a<c^-$ and $b>c^+$,
	\begin{align*}
	 &\cI(X_{[a,c]},Y_{[a,b]} \mid V) - \left(\cI(X_{[a,c]} \mid V)+\cI(Y_{[c,b]} \mid V)\right) \\&\qquad\qquad\qquad\qquad= \cI(X_{[c^-,c]},Y_{[c,c^+]} \mid V) - \left( \cI(X_{[c^-,c]} \mid V) + \cI(Y_{[c,c^+] } \mid V) \right) .
    \end{align*}
The right-hand side is bounded above in absolute value by $Z_c := \cI((X,Y)_{[c^-,c^+]} \mid V)$. Thus, \eqref{eq:SMB_X_Y_sup_c} will follow by showing that $Z_n = o(|n|)$ as $|n|\to \infty$ almost surely.
	Define $\tilde Z_n := Z_n M_n$, and note that $Z_n = \tilde Z_{n^-}$. Also, $n-n^-=o(|n|)$ almost surely. Thus, it suffices to show that $\tilde Z_n = o(|n|)$ almost surely.
	Using Kac's lemma (in the second equality below),
	\begin{align*}
	 \E \tilde Z_0 = \E[Z_0 M_0] &= \E[\cI(X_0,Y_0 \mid V, (X,Y)_{[0^-,0)})] \\&= H(X_0,Y_0 \mid V, (X,Y)_{[0^-,0)}) \le H(X_0,Y_0 \mid V) < \infty .
	\end{align*}
 It follows from the pointwise ergodic theorem that $\lim_{|n| \to \infty} \frac{\tilde Z_n}{|n|}=0$ almost surely. This proves~\eqref{eq:SMB_X_Y_sup_c}.


	Similarly (or simply by applying~\eqref{eq:SMB_X_Y_sup_c} with $(X,X)$ in place of $(X,Y)$),
	\begin{equation}\label{eq:SMB_X_sup}
		\left| \cI(X_{[i,j]} \mid V) - \left(\cI(X_{[i,0]} \mid V)+\cI(X_{[0,j]} \mid V)\right)\right| = o(j-i)  \qquad\text{as }j-i \to \infty\text{ with }i<0<j.
	\end{equation}
	In fact, the left-hand side is almost surely bounded over $i<0<j$, but we will not need this.
	By the relative Shannon--McMillan--Breiman theorem, almost surely,
	\begin{align*}
		\cI(X_{[0,j]} \mid V) &= jh(X \mid V) + o(j) &&\text{as }j \to \infty,\\
		\cI(X_{[i,0]} \mid V) &= |i|h(X \mid V) + o(|i|) &&\text{as }i \to -\infty.
	\end{align*}
	Thus, together with~\eqref{eq:SMB_X_sup}, we obtain that, almost surely,
	\[ \cI(X_{[i,j]} \mid V)  = (j-i)h(X \mid V) + o(j-i) \qquad\text{as }j-i \to \infty\text{ with }i<0<j.\]
	It follows from this that, almost surely,
	\begin{equation}\label{eq:SMB_X_a_b}
		\cI(X_{[a,c]} \mid V) = (c-a)h(X \mid V)+ o(|a|+|c|) \qquad\text{as }|a|+|c| \to \infty \text{ with }a<c .
	\end{equation}
	Similarly, almost surely,
	\begin{equation}\label{eq:SMB_Y_a_c}
		\cI(Y_{[c,b]} \mid V) = (b-c)h(Y \mid V)+ o(|b|+|c|) \qquad\text{as }|b|+|c| \to \infty \text{ with }c<b .
	\end{equation}
	Combining \eqref{eq:SMB_X_Y_sup_c} with \eqref{eq:SMB_X_a_b} and \eqref{eq:SMB_Y_a_c} yields the lemma.
\end{proof}

 Let $X$ and $\tilde X$ be processes taking values in $\cA$ and $\tilde\cA$, respectively, which are both discrete countable sets, and let $W$ be a process taking values in an arbitrary Polish space $\cB$.
A \textbf{partial isomorphism} of $X$ and $\tilde X$ relative to $W$  is a pair $(\pi,\tilde \pi)$ such that:
\begin{itemize}
 	\item $\pi: \cA^\Z \times \cB^\Z \to (\tilde\cA \cup \{\star\})^\Z$ is an equivariant map.
 	\item $\tilde\pi: \tilde\cA^\Z \times \cB^\Z \to (\cA \cup \{\star\})^\Z$ is an equivariant map.
 	\item There is a joining of $(X,W)$ and $(\tilde X,W)$ relatively to $W$ so that $\pi(X,W)$ is a partial process of $\tilde X$ and $\tilde \pi(\tilde X,W)$ is a partial process of $X$.
\end{itemize}
We say that $(\pi,\tilde \pi)$ is \textbf{finitary} if both $\pi$ and $\tilde \pi$ are finitary maps (with respect to the distributions of $(X,W)$ and $(\tilde X,W)$, respectively). We say that $(\pi',\tilde \pi')$ \textbf{extends} $(\pi,\tilde \pi)$ if 
	$\pi(X,W) \preceq \pi'(X,W)$ and $\tilde \pi(\tilde X,W) \preceq \tilde \pi'(\tilde X,W)$ almost surely. Recall that the uncertainty of a partial process $Y$ is $\Pr(Y=\star)$.
	
	The following lemma shows that an ``extending'' sequence of finitary partial (relative) isomorphisms  with vanishing uncertainties witnesses the existence of a (relative) finitary isomorphism:
		
	\begin{lemma}\label{lem:limit_of_partial_isomorphisms}
		Let $(X,W)$ and $(\tilde{X},W)$ be two ergodic joinings.
		Suppose there exists a sequence $(\pi_i,\tilde\pi_i)_{i=1}^\infty$ of finitary partial isomorphisms of $X$ and $\tilde{X}$ relative to $W$ which extend one another, such that the uncertainties of $\pi_i(X,W)$ and $\tilde \pi_i(\tilde X,W)$ both tend to $0$ as $i \to \infty$. Then $X$ and $\tilde{X}$ are finitarily isomorphic relatively to $W$.
	\end{lemma}
	\begin{proof}
		Define $\pi: \cA^\Z \times \cB^\Z \to (\tilde \cA \cup \{\star\})^\Z$ by
		\[ \pi(x,w)_n := \lim_{i \to \infty} \pi_i(x,w)_n = \begin{cases} \pi_{I(x,w)_n}(x,w) &\text{if }I(x,w)_n<\infty \\ \star &\text{otherwise} \end{cases} ,\]
		where $I(x,w)_n :=  \min \{ i : \pi_i(x,w)_n \ne \star\}$. It is straightforward to check that $\pi(X,W)$ is a partial process of $\tilde{X}$. Moreover, since $\Pr( I(X,W)_n \ge i) \to 0$, it follows that, almost surely, $I(X,W)_n < \infty$ and hence $\pi(X,W) \in \tilde \cA^\Z$.
		Similarly, one defines $\tilde\pi: \tilde\cA^\Z \times \cB^\Z \to (\cA \cup \{\star\})^\Z$ and has that $\tilde{\pi}(\tilde X,W) \in \cA^\Z$ almost surely. In addition, one may check that $\tilde\pi(\pi(X,W),W) = X$ and $\pi(\tilde\pi(\tilde X,W),W) = \tilde X$ almost surely, so that $\pi$ is an isomorphism from $X$ to $\tilde{X}$ relative to $W$, and $\tilde\pi$ is its inverse. Finally, since each $\pi_i$ is a finitary map, so is $I$, and it follows that $\pi$ is a finitary map, and similarly for $\tilde{\pi}$.
	\end{proof}

We now introduce a number of ad-hoc definitions that are intended to give some structure to the notion of finitary partial isomorphism beyond the basic properties above, which will be useful for our purposes.
Let $W$ be an ergodic aperiodic process, let $K$ be a $\N$-valued process which is a finitary factor of $W$, let $M$ be a marker process which is a finitary factor of $W$, and let $X$ and $\tilde X$ be two processes which are finitarily $K$-dependent relative to $W$. All partial isomorphisms below are partial isomorphisms of $X$ and $\tilde X$ relative to $W$, and we do not explicitly write this.

We say that an equivariant map $\pi\colon \cA^\Z \times \cB^\Z \to \cC^\Z$ is an \textbf{$M$-block code (relative to $W$)} if $\pi(X,W)_n$ depends on $X$ only through $X_{I^M_n}$ almost surely. We say that an $M$-block code $\pi$ is \textbf{finitary} if its dependence on $W$ is finitary. More precisely, if the factor map $W \mapsto (f_n)_n$ is finitary, where $f_n$ is the element of $(\cC^I)^{\cA^I}$ defined by $I:=I^M_n$ and $f_n(x) := \pi(x,W)_{I^M_n}$ (which is almost surely well defined since $\pi$ is an $M$-block code).
Note that finitary $M$-block codes preserve $M$-block-dependence and that compositions of finitary $M$-block codes are again finitary $M$-block codes.
A process is a \textbf{finitary $M$-block factor of $X$ (relative to $W$)} if it is the image of $(X,W)$ under a finitary $M$-block code, and we denote this by $X \hookrightarrow_M X'$ (omitting $W$ from the notation).

A partial isomorphism $(\pi,\tilde\pi)$ is \textbf{$(M,X',\tilde X')$-adapted} (see \cref{fig:partial-iso}) if
\begin{itemize}
	\item $X \hookrightarrow_M X' \hookrightarrow_M \pi(X,W)$ and $X' \preceq X$.
	\item $\tilde X \hookrightarrow_M \tilde X' \hookrightarrow_M \tilde\pi(\tilde X,W)$ and $\tilde X' \preceq \tilde X$.
	\item There is a joining of $(X,W)$ and $(\tilde X,W)$ relatively to $W$ such that $\pi(X,W) \preceq \tilde X'$ and $\tilde \pi(\tilde X,W) \preceq X'$.
\end{itemize}
This notion will allow to keep track and control on the partial isomorphisms we construct through the processes $(M,X',\tilde X')$ and the following notion of complexity.
The \textbf{$M$-complexity (relative to $W$)} of an $M$-block-dependent process $Y$ is at most $\kappa \ge 0$ if almost surely, given $W$, there are at most $e^{\kappa |I^M_0|}$ values that the random variable $Y_{I^M_0}$ can take. 


\begin{figure}
 \includegraphics[scale=1]{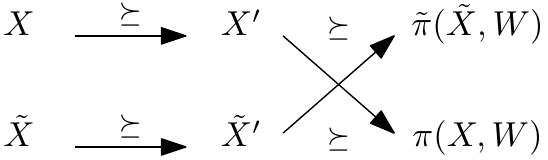}
 \caption{An illustration of the relations in an $(M,X',\tilde X')$-adapted partial isomorphism. The arrows represent finitary $M$-block factors.}
 \label{fig:partial-iso}
\end{figure}

	

The proof of \cref{thm:relative_smorodinsky_k_dependent} via \cref{lem:limit_of_partial_isomorphisms} proceeds by the following ``improvement step'', whose essence is due to  Keane and Smorodinsky \cite{keane1977class}. The basic idea was subsequently used by Keane and Smorodinsky to prove finitary isomorphism of equal-entropy Bernoulli schemes \cite{keane1979bernoulli} and Markov chains \cite{keane1979finitary}. Roughly speaking, the lemma says two things: First, that we can reduce the uncertainty of $\pi$ to be almost as low as the imprecision of $\tilde\pi$. Second, we can reduce the imprecision of $\pi$ arbitrarily, at the expense of increasing its complexity. This will allow us to iterate this back and forth in order to eventually arbitrarily decrease the uncertainties of both $\pi$ and $\tilde\pi$.


\begin{lemma}[partial isomorphism improvement step]\label{lem:improve_partial_isomorphism}
	Let $X$ and $\tilde{X}$ be two processes which are finitarily $M$-block-dependent relative to $W$, where $M$ is a finitary factor of $W$.
	Suppose that $h(X \mid W)<\infty$ and that $\tilde{X}$ has $M$-complexity strictly less than $h(X \mid W)$.
	Let $(\pi,\tilde \pi)$ be an $(M,X',\tilde X)$-adapted partial isomorphism of $X$ and $\tilde X$ relative to $W$, where $X'$ has $M$-complexity strictly less than $h(X \mid W)$.
	Then for any $\epsilon>0$ there exists a  finitary dilution $M'$ of $M$,  a partial process $X''$ of  $M'$-complexity strictly less than $h(X \mid W)$ and uncertainty at most $\epsilon$ such that $X' \preceq X'' \preceq X$ and an $(M',X'',\tilde X)$-adapted partial isomorphism $(\pi',\tilde \pi)$ of $X$ and $\tilde X$ relative to $W$ that extends $(\pi,\tilde\pi)$ such that $\pi'(X,W)$ has uncertainty at most $\epsilon$.
\end{lemma}

\begin{proof} 
 Denote $h := h(X \mid W)$ and $h' := h(X' \mid W)$.
	Let $\kappa<h$ be such that $X'$ and $\tilde X$ have $M$-complexity at most $\kappa$.
The fact that $X'$ has complexity at most $\kappa$ implies that $h' \le \kappa$. Thus, we can choose $\kappa'$ and $\eta \in (0,\frac\epsilon4)$ so that
	\[ \kappa < (1-\eta)h + \eta h' < \kappa' < h .\]
	
		Given a sufficiently sparse finitary dilution $M'$ of $M$, we aim to construct a process $X''$, and a partial isomorphism $(\pi',\tilde\pi)$ that extends $(\pi,\tilde\pi)$ (note that $\tilde\pi$ remains unchanged). The process $X''$ will be chosen so that $X' \preceq X'' \preceq X$ and so that it has $M'$-complexity at most $\kappa'$ and uncertainty at most $\epsilon$. We will then construct $\pi'$ so that $(\pi',\tilde\pi)$ is $(M',X'',\tilde X)$-adapted and so that $\pi'(X,W)$ has uncertainty at most $\epsilon$. This will yield the required partial isomorphism $(\pi',\tilde\pi)$, thereby establishing the lemma.

	\medskip
	\noindent\textbf{The construction of $X''$.}
	Partition each interval $I^{M'}_n$ into a ``left'' interval $I^{M',\text{left}}_n$ and a ``right'' interval $I^{M',\text{right}}_n$ so that each is a union of $M$-blocks and so that the relative length of the right interval is as close as possible to $\eta$ from below, i.e., $I^{M',\text{left}}_n=[a,c)$ and $I^{M',\text{right}}_n=[c,b)$, where $I^{M'}_n=[a,b)$ and $c$ is the smallest integer such that $M_c=1$ and $b-c<\eta(b-a)$ (note that $c$ might equal $b$).
	Define an auxiliary partial process $X'''$ as follows:
		\[ X'''_n :=
	\begin{cases}
		X_n &\text{if }n \in I^{M',\text{left}}_n \\
		X'_n &\text{if }n \in  I^{M',\text{right}}_n
	\end{cases}.
\]
	Define $X''$ as follows:
	\[ X''_n :=
	\begin{cases}
	 X'''_n &\text{if }\mathcal{I}(X'''_{I^{M'}_n}  \mid W) \le |I^{M'}_n| \kappa' - \ln 2 \\
	 X'_n &\text{otherwise}
	\end{cases} .\]
	Note that $X \hookrightarrow_M X'$, together with the fact that $X$ is finitarily $M$-block-dependent relative to $W$, implies that $X \hookrightarrow_{M'} X''$. Note also that $X'' \hookrightarrow_M X'$.

	To complete the definition of $X''$, it remains to specify the choice of the marker process $M'$.
	Choose $h_1$ and $h_2$ such that $\kappa < h_1 < h_2 < (1-\eta)h+\eta h'$, and let $N$ be a sufficiently large integer so that
	\begin{equation}\label{eq:N}
	 e^{\kappa n} \le \tfrac{1}{2}e^{\kappa' n}, \qquad e^{-(h_1- \kappa)n} \le \tfrac{\epsilon}{2}, \qquad e^{-(h_2-h_1)n} \le \tfrac{\epsilon}4 \qquad\text{for all }n \ge N.
	\end{equation}
\cref{lem:Shannon-McMillan-bidirectional} implies that $\frac{1}{|I^{M'}_0|}\mathcal{I}(X'''_{I^{M'}_0}\mid W) \to (1-\eta)h + \eta h'$
almost surely as the density of  $M'$ tends to $0$ (along a fixed sequence of marker processes).
In particular, if $M'$ is sufficiently sparse, then
\begin{equation}\label{eq:iso-info-bound}
\Pr\Big( |I^{M'}_0| h_2 < \cI(X'''_{I^{M'}_0}\mid W) < |I^{M'}_0| \kappa' - \ln 2 \Big) \ge 1- \tfrac\epsilon8.
\end{equation}
By \cref{lem:low-density-markers}, $M$ contains a finitary dilution with arbitrarily large minimal gap length (and hence also with arbitrarily low density). We choose $M'$ to be a finitary dilution of $M$ for which~\eqref{eq:iso-info-bound} holds and
\begin{equation}\label{eq:min-gap}
|I^{M'}_0| \ge N \qquad\text{almost surely} .
\end{equation}
	
	\medskip
	\noindent\textbf{The complexity of $X''$.}
	Let us show that $X''$ has $M'$-complexity at most $\kappa'$.
	Observe that $X''_{I^{M'}_0}$ equals either $X'''_{I^{M'}_0}$ or $X'_{I^{M'}_0}$.
	We first claim that, given $W$, there are at most $\frac12 e^{\kappa'|I^{M'}_0|}$ possible values $x$ such that $X''_{I^{M'}_0} = X'_{I^{M'}_0}=x$ occurs with positive probability. Indeed, since $X'$ has $M$-complexity at most $\kappa$ (and since $M'$ is a dilution of $M$), the number of possible values for $X'_{I^{M'}_0}$ is at most $e^{\kappa |I^{M'}_0|}$, which is at most $\frac12 e^{\kappa' |I^{M'}_0|}$ by~\eqref{eq:N} and~\eqref{eq:min-gap}.
	It remains to show that, given $W$, there are at most the same number of other possible values $x$ such that $X''_{I^{M'}_0} = X'''_{I^{M'}_0}=x$ occurs with positive probability. This is immediate from the definition of $X''$.
This proves that $X''$ has $M'$-complexity at most $\kappa'$.

		
	\medskip
	\noindent\textbf{The uncertainty of $X''$.}
	Let us show that $X''$ has uncertainty at most $\epsilon$.
	By~\eqref{eq:iso-info-bound}, we have
	\begin{equation}\label{eq:imprecision-bound}
	\Pr\left(X''_0 \ne  X'''_0\right) \le  \Pr\left(X''_{I^{M'}_0} \ne  X'''_{I^{M'}_0}\right) \le \Pr\left(\cI( X'''_{I^{M'}_0}\mid W) > |I^{M'}_0| \kappa' - \ln 2\right) \le \tfrac\epsilon8.
	\end{equation}
	Also,
	\[ \Pr\left(X'''_0 = \star\right) \le \Pr\left(0 \in I^{M',\text{right}}_0\right) \le \eta \le \tfrac{\epsilon}{2}. \]
	Altogether we conclude that $\Pr(X''_0 = \star) \le \epsilon$.

	\medskip
	\noindent\textbf{The construction of $\pi'$.}
	We now turn to defining the map $\pi'$ which will yield the desired partial isomorphism $(\pi',\tilde\pi)$.
	By the assumption that $(\pi,\tilde\pi)$ is $(M,X',\tilde X)$-adapted, there exists a joining of $(X,W)$ and $(\tilde X,W)$ relative to $W$ such that $\pi(X,W) \preceq \tilde X$ and $\tilde \pi(\tilde X,W) \preceq X' \preceq X''$ almost surely.
We may further assume that this joining has the property that $(\cL((X'',\tilde X)_{I^{M'}_n} \mid W))_{n \in \Z}$ is a finitary factor of $W$. 
To see this, note first that $(\cL(X'')_{I^{M'}_n} \mid W))_{n \in \Z}$ is a finitary factor of $W$ (this follows from the facts that $X$ is finitarily $M'$-block-dependent, $M'$ is a finitary factor of $W$, and $X \hookrightarrow_{M'} X''$). Similarly, $(\cL(\tilde X)_{I^{M'}_n} \mid W))_{n \in \Z}$ is a finitary factor of $W$. Thus, if needed, we may replace the original joining by a new one in which the conditional coupling of $(X'',\tilde X)_{I^{M'}_n}$ given $W$ is chosen to be $\cL((X'',\tilde X)_{I^{M'}_n} \mid W_{[-R,R]})$ (with respect to the original joining), where $R$ (a stopping time, measurable with respect to $W$) is large enough to determine the conditional marginal distributions of $X''_{I^{M'}_n}$ and $\tilde X_{I^{M'}_n}$.
This joining has the claimed property.

   Let $U_W$ denote the set of admissible values of $\tilde X_{I^{M'}_0}$ given $W$ and let $V_W$ denote the set of admissible values of $X''_{I^{M'}_0}$ given $W$. Let $\rho_W$ and $\sigma_W$ denote the probability measures on $U_W$ and $V_W$, respectively, corresponding to the conditional distributions of $\tilde X_{I^{M'}_0}$ and $X''_{I^{M'}_0}$ given $W$. 
  	 Let $\lambda_W$ be the coupling of $\rho_W$ and $\sigma_W$ induced by the joint distribution of $(\tilde X_{I^{M'}_0},X''_{I^{M'}_0})$ given $W$ under the above joining. Note that the function $W \mapsto \lambda_W$ is measurable, equivariant and finitary. 

   We apply \cref{lem:coupling_reduction} to deduce that almost surely there exists a coupling $\tilde \lambda_W$ of $(U_W,\rho_W)$ and $(V_W,\sigma_W)$ which is absolutely continuous with respect to $\lambda_W$ and satisfies that
\begin{equation}\label{eq:committed}
\left|\left\{v \in V_W : v \mbox{ is not } \tilde\lambda_W \mbox{-committed} \right\} \right| \le |U_W|-1.
\end{equation}
   Moreover, the proof of \cref{lem:coupling_reduction} implicitly describes an ``algorithm'' which given a coupling $\lambda_W$ outputs $\tilde \lambda_W$ as above. Thus there is a Borel measurable function  $\lambda_W \mapsto \tilde \lambda_W$ satisfying the above.
   We can furthermore arrange that the map $W \mapsto \tilde \lambda_W$ will be equivariant.
Altogether, the function $W \mapsto \tilde \lambda_W$ is measurable, equivariant and finitary.

Given an interval $I \subset \Z$ containing $0$, and a probability measure $\lambda$ on $U \times V \subset \tilde \cA^I \times \cA^I$, let 
	$\Phi_\lambda:V \to \tilde \cA \cup \{\star\}$ be the function that returns $\star$ whenever $v$ does not uniquely determine $u_0$ up to a $\lambda$-null set, and the uniquely determined value of $u_0$ otherwise. More precisely,
	\[\Phi_\lambda(v) := \begin{cases}
		\tilde a&\text{if }\lambda \left( \left\{ (u,v) : u \in U,~ u_{0} = \tilde a   \right\}\mid \left\{ (u,v) : u \in U \right\} \right) =1 \\
		\star &\text{otherwise}
		\end{cases}.
		\]
Strictly speaking, this does not define $\Phi_\lambda(v)$ when $\lambda(U \times \{v\})=0$, in which case we set $\Phi_\lambda(v):=\star$.

Define an equivariant map $\pi':\cA^\Z \times \cB^\Z \to (\tilde \cA \cup \{\star\})^\Z$ by
\[ \pi'(X,W)_0 := \Phi_{\tilde \lambda_W}(X''_{I_0^{M'}}).\]
Since $X'' \hookrightarrow_{M} X' \hookrightarrow_{M} \pi(X,W) \preceq \tilde X$, we have that $\Phi_{\lambda_W}(X''_{I_0^{M'}}) = \pi(X,W)_0$ almost surely whenever $\pi(X,W)_0 \ne \star$. 
 Since $\tilde \lambda_W$ is absolutely continuous with respect to $\lambda_W$,  in this case we have $\Phi_{\tilde \lambda_W}(X''_{I_0^{M'}}) = \pi(X,W)_0$, so
$\pi'(X,W)_0 = \pi(X,W)_0$. 
This verifies that $\pi'(X,W)$ almost surely extends $\pi(X,W)$.

\medskip
\noindent\textbf{The adaptedness of the partial isomorphism.}
 	We show that $(\pi',\tilde\pi)$ is $(M',X'',\tilde X)$-adapted.
	To this end, we first need to show that $X \hookrightarrow_{M'} X'' \hookrightarrow_{M'} \pi'(X,W)$. Indeed, we have already seen that $X \hookrightarrow_{M'} X''$, and $X'' \hookrightarrow_{M'} \pi'(X,W)$ follows easily from the definition of $\pi'$.
It remains to demonstrate the existence of a joining of $(X,W)$ and $(\tilde X,W)$ such that $\pi'(X,W) \preceq \tilde X$ and $\tilde\pi(\tilde X,W) \preceq X''$.
We set the joint conditional distribution of $X''$ and $\tilde X$ on $I^{M'}_0$ given $W$ according to $\tilde\lambda_W$, and arbitrarily extend this to a joining.
  The fact that $\tilde \lambda_W$ is almost surely absolutely continuous with respect to $\lambda_W$ implies that $\tilde \pi(\tilde X,W) \preceq X' \preceq X''$ with respect to this joining.
  The definition of $\pi'$ together with the fact that $\tilde \lambda_W$ is supported on $U_W \times V_W$ shows that $\pi'(X,W) \preceq \tilde X$ with respect to this joining. 
  
    
    \medskip
	\noindent\textbf{The uncertainty of $\pi'$.}
     It remains to show that $\pi'(X,W)$ has uncertainty at most $\epsilon$. Observe that $\pi'(X,W)_{I^{M'}_0} = \tilde X_{I^{M'}_0}$ on the event that $X''_{I^{M'}_0}$ is $\tilde \lambda_W$-committed. In particular,
\[
\Pr(\pi'(X,W)_0 = \star \mid W) \le \sigma_W\left( \left\{ v \in V_W : v \mbox{ is not } \tilde \lambda_W \mbox{-committed} \right\}  \right) .\]
   Note that by~\eqref{eq:committed}, almost surely,
   \[
   \sigma_W\left( \left\{ v \in V_W : v \mbox{ is not } \tilde \lambda_W \mbox{-committed} \right\}  \right) \le \Pr\left(\cI( X''_{I^{M'}_0} \mid W) \le |I^{M'}_0|h_1 \mid W \right) + e^{-h_1|I^{M'}_0|} |U_W|.
   \]
   Thus, taking expectation over $W$, we get that
   \[ \Pr(\pi'(X,W)_0 = \star) \le \Pr\left(\cI( X''_{I^{M'}_0} \mid W) \le |I^{M'}_0|h_1 \right) + \E\left[e^{-h_1|I^{M'}_0|} |U_W|\right]  .\] 
   Let us bound each of the two terms on the right-hand side.
   For the first term, we have by \cref{lem:info-comparison-bound} (applied conditionally on $W$ and then taking expectation over $W$) and~\eqref{eq:min-gap},
   \begin{align*} 
   \Pr\left(\cI( X''_{I^{M'}_0} \mid W) \le |I^{M'}_0|h_1\right) &\le \Pr\left(X''_{I^{M'}_0} \neq X'''_{I^{M'}_0}\right) + \Pr\left(\cI( X'''_{I^{M'}_0}\mid W) \le |I^{M'}_0| h_2\right) + e^{-(h_2-h_1)N} \le \tfrac\epsilon2 ,
   \end{align*}
   where the second inequality follows from~\eqref{eq:N}, \eqref{eq:iso-info-bound} and~\eqref{eq:imprecision-bound}.
   For the second term, using that $\tilde X$ has $M$-complexity at most $\kappa$ (and since $M'$ is a dilution of $M$), it follows that
   $|U_W| \le e^{\kappa |I^{M'}_0|}$ almost surely. Hence, using~\eqref{eq:N} and~\eqref{eq:min-gap}, we see that $e^{-h_1|I^{M'}_0|}|U_W| \le e^{-(h_1-\kappa)|I^{M'}_0|} \le \tfrac\epsilon2$ almost surely. 
   Putting these bounds together, we conclude that $\Pr(\pi'(X,W)_0 = \star) \le \epsilon$.	
\end{proof}

We are now ready to prove \cref{thm:relative_smorodinsky_k_dependent}.
The idea is to apply ``ping-pong'' iterations of \cref{lem:improve_partial_isomorphism} to obtain a sequence of better and better finitary partial isomorphisms and then apply \cref{lem:limit_of_partial_isomorphisms}. In fact, this idea can also be used to obtain \cref{thm:pro-dependent}, but with the aim of making the proof easier to digest, we instead prove \cref{thm:relative_smorodinsky_k_dependent} first, and then use it to deduce the stronger \cref{thm:pro-dependent} in \cref{sec:pro-dependent}.

Given a process $X$, an $\N$-valued process $K$, a marker process $M$, and a subset $A$ of the alphabet of $X$, we define a partial process $X^{(K,M,A)}$ of $X$ by
	\begin{equation}\label{eq:X-K-M}
	X^{(K,M,A)}_n := \begin{cases}
		X_n & \text{if } X_n \in A\text{ and } [n-K_n,n+K_n] \subseteq I^M_n\\
		\star & \mbox{otherwise}
	\end{cases}.
	\end{equation}
Note that if $K$ and $M$ are finitary factors of $W$, then $X^{(K,M,A)}$ is a finitary $M$-block factor of $X$ relative to $W$, and if, in addition, $X$ is finitarily $K$-dependent relative to $W$, then $X^{(K,M,A)}$ is also finitarily $M$-block-dependent relative to $W$.
Recall also that finitary $M$-block codes preserve finitary $M$-block-dependence and that compositions of finitary $M$-block codes are again finitary $M$-block codes.

\begin{proof}[Proof of \cref{thm:relative_smorodinsky_k_dependent}]
Denote $h:=h(X \mid W)=h(\tilde X \mid W)$. Let $K$ and $\tilde K$ be finitary factors of $W$ such that $X$ and $\tilde X$ are finitarily $K$- and $\tilde K$-dependent relative to $W$, respectively.

We will prove by induction the existence of a sequence $(M^{(n)},A_n,\tilde A_n,X^{(n)},\tilde X^{(n)},\pi_n,\tilde\pi_n)_{n=1}^\infty$ such that
\begin{itemize}
 \item $M^{(n)}$ are marker processes which are finitary factors of $W$ that dilute one another.
 \item  $A_n$ are finite sets increasing to $\cA$, and $\tilde\cA_n$ are finite sets increasing to $\tilde\cA$. 
  \item $X^{(n)}$ is a finitary $M^{(n)}$-block factor and partial process of $X^{(K,M^{(n)},A_n)}$.
  \item $\tilde X^{(n)}$ is a finitary $M^{(n)}$-block factor and partial process of $\tilde X^{(\tilde K,M^{(n)},\tilde A_n)}$.
  \item Each of $X^{(n)}$ and $\tilde X^{(n)}$ has $M^{(n)}$-complexity strictly less than $h$ and uncertainty at most $\frac1n$.
 \item $(\pi_n,\tilde\pi_n)$ is an $(M^{(n)},X^{(n)},\tilde X^{(n)})$-adapted partial isomorphism of $X$ and $\tilde X$ which extends $(\pi_{n-1},\tilde\pi_{n-1})$, such that the uncertainties of $\pi_n(X,W)$ and $\tilde\pi_n(\tilde X,W)$ are at most $\frac3n$.
 \end{itemize}
The theorem will then follow from \cref{lem:limit_of_partial_isomorphisms}.


Let $M^{(1)}$ be any marker process which is a finitary factor of $W$, set $A_1=\tilde A_1= \emptyset$ and $X^{(1)}=\tilde X^{(1)}\equiv \star$. Let $(\pi_1,\tilde\pi_1)$ be the trivial partial isomorphism of $X$ and $\tilde X$ which equals $\star$ everywhere.

Suppose we have defined $(M^{(n)},A_n,\tilde A_n,X^{(n)},\tilde X^{(n)},\pi_n,\tilde\pi_n)$. The construction for $n+1$ consists of two steps (``ping'' and ``pong''), one to improve $\pi_n$ and another to improve $\tilde\pi_n$.
The first step proceeds as follows:
Let $\kappa<h$ be such that both $X^{(n)}$ and $\tilde X^{(n)}$ have $M^{(n)}$-complexity at most~$\kappa$.
Let $M$ be a sparse enough finitary dilution of $M^{(n)}$. Let $A=A_{n+1}$ be a large enough finite subset of $\cA$ containing $A_n$ so that $h' := h(X^{(K,M,A)} \mid W) > \kappa$ and so that $X^{(K,M,A)}$ has uncertainty strictly less than $\frac1{n+1}$.
Toward applying \cref{lem:improve_partial_isomorphism}, set $Y:=X^{(K,M,A)}$, $Y' := X^{(n)}$ and $\tilde Y := \tilde X^{(n)}$.
Since $X$ is finitarily $K$-dependent relative to $W$, and $K$ and $M$ are finitary factors of $W$, it follows that $Y$ is finitarily $M$-block-dependent relative to $W$.
 Also, note that $(\pi_n,\tilde\pi_n)$ can be seen as an $(M,Y',\tilde Y)$-adapted partial isomorphism of $Y$ and $\tilde Y$, and that $\tilde Y$ has uncertainty strictly less than $\frac3{n+1}$.
Thus, \cref{lem:improve_partial_isomorphism} yields a marker process $M'$ which is a finitary dilution of $M$, a partial process $X^{(n+1)}$ of $M'$-complexity strictly less than $h'$ and uncertainty at most $\frac1{n+1}$ such that $X^{(n)} \preceq X^{(n+1)} \preceq X^{(K,M,A)}$ and an $(M',X^{(n+1)},\tilde X^{(n)})$-adapted partial isomorphism $(\pi_{n+1},\tilde \pi_n)$ of $X^{(K,M,A)}$ and $\tilde X^{(n)}$ which extends $(\pi_n,\tilde\pi_n)$ and such that $\pi_{n+1}(X,W)$ has uncertainty at most $\frac3{n+1}$ (note that when applying \cref{lem:improve_partial_isomorphism}, we regard $Y$ and $\tilde Y$ as regular processes, and not as partial processes).

The second step is similar, with the roles of $X$ and $\tilde X$ reversed:
Let $\kappa'<h$ be such that both $\tilde X^{(n)}$ and $X^{(n+1)}$ have $M'$-complexity at most $\kappa'$.
Let $M''$ be a sparse enough finitary dilution of $M'$. Let $\tilde A=\tilde A_n$ be a large enough finite subset of $\tilde\cA$ containing $\tilde A_n$ so that $h'' := h(\tilde X^{(\tilde K,M'',\tilde A)} \mid W) > \kappa'$ and so that $\tilde X^{(\tilde K,M'',\tilde A)}$ has uncertainty strictly less than $\frac1{n+1}$.
Toward applying \cref{lem:improve_partial_isomorphism}, set $Z:=\tilde X^{(\tilde K,M'',\tilde A)}$, $Z' := \tilde X^{(n)}$ and $\tilde Z := X^{(n+1)}$, and note that $(\tilde\pi_n,\pi_{n+1})$ can be seen as an $(M'',Z',\tilde Z)$-adapted partial isomorphism of $Z$ and $\tilde Z$, and that $\tilde Z$ has uncertainty at most $\frac1{n+1}$.
Thus, \cref{lem:improve_partial_isomorphism} yields a marker process $M^{(n+1)}$ which is a finitary dilution of $M''$, a partial process $\tilde X^{(n+1)}$ of $M^{(n+1)}$-complexity strictly less than $h''$ and uncertainty at most $\frac1{n+1}$ such that $\tilde X^{(n)} \preceq \tilde X^{(n+1)} \preceq \tilde X^{(\tilde K,M'',\tilde A)}$ and an $(M^{(n+1)},\tilde X^{(n+1)},X^{(n+1)})$-adapted partial isomorphism $(\tilde \pi_{n+1},\pi_{n+1})$ of $\tilde X^{(\tilde K,M'',\tilde A)}$ and $X^{(n+1)}$ which extends $(\tilde\pi_n,\pi_{n+1})$ and such that $\tilde\pi_{n+1}(X,W)$ has uncertainty at most $\frac3{n+1}$.
Finally, note that $(\pi_{n+1},\tilde \pi_{n+1})$ can be seen as an $(M^{(n+1)},X^{(n+1)},\tilde X^{(n+1)})$-adapted partial isomorphism of $X$ and $\tilde X$. This completes the induction step.
\end{proof}

\subsection{Proof of \cref{thm:pro-dependent}}\label{sec:pro-dependent}
We now show how to use \cref{thm:smorodinsky_k_dependent,thm:relative_smorodinsky_k_dependent} to deduce \cref{thm:pro-dependent}.
For this, we require an additional result about finitary isomorphisms of \iid\ processes taking values in the space of sequences $\N^\N$.
As we have mentioned, Keane and Smorodinsky showed that any two finite-valued \iid\ processes of equal entropy are finitarily isomorphic. 
\cref{thm:smorodinsky_k_dependent} implies that this is also true for \iid\ processes taking values in a discrete countable space (recall that the two notions of finitary discussed in \cref{sec:defs_notation} are equivalent in this case), where the infinite-entropy case was already proved by Petit~\cite{petit1982deux}.
The following result (\cref{thm:top-finitary} below) extends this to processes taking values in Polish spaces which are not necessarily discrete.

We first mention a simple, general result about embedding Polish spaces off null sets:
\begin{prop}\label{prop:top_embed_off_null}
	Let $M$ be a Polish space equipped with a Borel probability measure. Then there exists a  set $M'\subset M$ of full measure that can be topologically embedded in $\{0,1\}^\N$. 
\end{prop}
\begin{proof}
	Let $d:M \times M \to \mathbb{R}_+$ be a  metric on $M$ which is compatible with its Polish topology.
	Choose a countable dense subset $D \subseteq M$. For $\lambda >0$ and $y \in M$, define
	\[M_{\lambda,y} :=\{ x \in M:~ d(x,y) \in \lambda\mathbb{Q}\}.\]
	Then for every $y \in M$, the set $M_{\lambda,y}$ has zero probability for Lebesgue almost every $\lambda$. It follows that for Lebesgue almost every $\lambda$, the set
	$M_\lambda := \bigcap_{y \in D} M_{\lambda,y}$ has zero probability. 
	To complete the proof, it suffices to show that $M \setminus M_\lambda$ can be topologically embedded in $\{0,1\}^\N$ for any $\lambda$. Towards this goal, fix $\lambda$ and consider the function 
	$\Phi:M\setminus M_\lambda \to (\mathbb{R}_+ \setminus \mathbb{Q})^{D}$ given by
	\[\Phi(x)_y := \lambda^{-1} d(x,y). \]
	Clearly $\Phi$ is continuous. We will show that $\Phi$ is a homeomorphism onto its image $\Phi(M \setminus M_\lambda)$ as follows.
	Let $M^*$ denote the space of closed subsets of $M$, equipped with the Fell topology (that is, the topology induced by the Hausdorff distance).
	The space $M$ naturally embeds in $M^*$ via the map $x \mapsto \{x\}$. 
	Define a function 
	$\Psi:  (\mathbb{R}_+ \setminus \mathbb{Q})^{D} \to M^*$ by
	\[\Psi(v) := \bigcap_{y \in D} \left\{ x \in M:~ d(x,y)= v_y\right\}.\]
	Again, it is straightforward  to check that $\Psi$ is continuous.
	Let us check that every $x \in M\setminus M_\lambda$ satisfies $\Psi(\Phi(x))=\{x\}$: By the density of $D$, there exists a sequence $(y_n)_{n \in \N}$ of elements of $D$ such that $d(x,y_n) \to 0$ so $\lim_{n \to \infty}\Phi(x)_{y_n} =0$. Then clearly $\Psi(\Phi(x))$ consists of accumulation points of $\{y_n\}_{n=1}^\infty$, so $\Psi(\Phi(x))=\{x\}$.
	We have shown that $M \setminus M_\lambda$ is homeomorphic to a subset of $(\mathbb{R}_+ \setminus \mathbb{Q})^{D} \cong (\mathbb{R}_+ \setminus \mathbb{Q})^\N$.
	Now $\mathbb{R}_+ \setminus \mathbb{Q}$ is clearly homeomorphic to a subset of $[0,1]\setminus \mathbb{Q}$. The latter is a homeomorphic to a subset of the trinary cantor set (by sending $x = \sum_{k=1}^\infty 2^{-k}x_k \in [0,1]\setminus \mathbb{Q}$ to $\sum_{k=1}^\infty3^{-k}x_k$, where $x_1,x_2,\ldots \in \{0,1\}$ are the digits in the binary expansion of $x$). This shows that $M\setminus M_\lambda$ is homeomorphic to a subset of $(\{0,1\}^\mathbb{N})^\mathbb{N} \cong \{0,1\}^\N$.
\end{proof}


\begin{thm}\label{thm:top-finitary}
Any two equal-entropy \iid\ $\Z$-processes taking values in Polish spaces are 
	 topo-finitarily isomorphic. 
	That is, if $X$ and $Y$ are equal entropy processes taking values in Polish spaces $\cA$ and $\cB$ respectively, then there is an equivariant measurable function $\phi: \cA^\Z \to \cB^\Z$ that restricts to a homeomorphism between sets of full measure
	and 
	such that $\phi(X)$ has the distribution of $Y$.
\end{thm}
\begin{proof}
Let $X$ be an \iid\ process taking values in a Polish space $\cA$. By \cref{prop:top_embed_off_null}, there exists a set $\cA' \subset \cA$ such that $\Pr(X_0 \in \cA')=1$ and which can be continuously embedded in $\{0,1\}^\N$, so by applying this embedding we can assume that $\cA \subset \{0,1\}^\N$.

For processes taking values in discrete sets, topo-finitary isomorphism coincides with stop-finitary isomorphism. Since any two equal-entropy \iid\ processes taking values in discrete countable spaces are finitarily isomorphic (by \cref{thm:smorodinsky_k_dependent}), to complete the proof, it suffices to show that $X$ is topo-finitarily isomorphic to some \iid\ process taking values in a discrete countable set.
 Specifically, we will show that $X$ is topo-finitarily isomorphic to an \iid\ process $Y$ taking values in the discrete set $\cB \subset (\N \cup \{*\})^\N$ consisting of sequences which have a finite prefix of numbers followed by an infinite sequence of stars, i.e.,
\[ \cB = \left\{ y \in (\N \cup \{*\})^\N : \text{there exists }i \in \N\text{ such that } y_1,\dots,y_{i-1} \in \N\text{ and } y_i = y_{i+1} = \cdots = *\right\}.\]

We may assume that $H(X_{0,t} \mid X_{0,1},\dots,X_{0,t-1})>0$ for all $t \in \N$, since $X$ is trivially finitarily isomorphic to the \iid\ process $((X_{n,t})_{t \in \cT})_{n \in \Z}$, where $\cT \subset \N$ consists of those $t$ where this entropy is positive, and since $X$ itself is already a finite-valued \iid\ process in the case that $\cT$ is finite.
To specify the process $Y$, we first let $(Y_{0,t})_{t \in \N}$ be a sequence of finite-valued random variables taking values in $\N \cup \{*\}$ such that:
\begin{itemize}
	\item $H(Y_{0,t} \mid Y_{0,1},\dots,Y_{0,t-1})=H(X_{0,t} \mid X_{0,1},\dots,X_{0,t-1})$.
	\item $\Pr(Y_{0,t}=* \mid Y_{0,t-1}=*)=1$.
	\item $\Pr(Y_{0,t}=* \mid Y_{0,1},\dots,Y_{0,t-1}) \ge 1/2$ almost surely.
\end{itemize} 
The existence of such a sequence of random variables is easily verified by induction on $t$. Let $Y$ be the \iid\ process having $Y_0=(Y_{0,t})_{t \in \N}$.
The second two properties above together assure that $Y_0 \in \cB$ almost surely.

Denote $X^{(t)}=(X_{n,1},\dots,X_{n,t})_{n \in \Z}$ and $Y^{(t)}=(Y_{n,1},\dots,Y_{n,t})_{n \in \Z}$. By induction on $t$, we prove the existence of a finitary isomorphism $\pi_t:(\{0,1\}^t)^\Z \to ((\N \cup \{*\})^t)^\Z$ between $X^{(t)}$ and $Y^{(t)}$. Each will extend the previous in the sense that $\pi_{t+1}(X^{(t+1)})_{0,i}=\pi_t(X^{(t)})_{0,i}$ and $\pi_{t+1}^{-1}(Y^{(t+1)})_{0,i} = \pi_t^{-1}(Y^{(t)})_{0,i}$ for all $1 \le i \le t$. 
Since $X^{(1)}$ and $Y^{(1)}$ are equal-entropy finite-valued \iid\ processes, the base case follows from the Keane--Smorodinsky finitary isomorphism theorem (which is a particular case of \cref{thm:smorodinsky_k_dependent}). The induction step follows using \cref{thm:relative_smorodinsky_k_dependent} and the induction hypothesis. 
It is then straightforward that the limit $\pi := \lim_{t \to \infty} \pi_t$ is well defined and is a topo-finitary isomorphism between $X$ and $Y$.
\end{proof}

\begin{proof}[Proof of \cref{thm:pro-dependent}]
%
%
Let $X$ be a countable-valued process  that is finitarily pro-dependent relative to $W$. 
	We will construct an $\N^\N$-valued \iid\ process $Y = (Y_n)_{n \in \Z}=((Y_{n,j})_{j=1}^\infty)_{n \in \Z}$, independent of $W$, and a relative isomorphism between $(X,W)$ and $(Y,W)$, such that the following holds\footnote{The property satisfied by this isomorphism is something of a mixture between the notion of topo-finitary and stop-finitary. It can be also be expressed as a stop-finitary isomorphism with respect to a certain filtration.}: Almost surely there exists a finite $N$ such that $X_0$ is determined by $(W_{[-N,N]},(Y_{n,j})_{|n|\le N,1\le j\le N })$ and almost surely for every $j \in \N$ there a finite $N_j$ such that $Y_{0,j}$ is determined by $(W,X)_{[-N_j,N_j]}$. In light of \cref{thm:top-finitary}, this will yield the theorem.

Let $(X^{(i)})$ be a sequence of partial processes increasing to $X$ as in the definition of finitarily pro-dependent. Let $\bar X$ be the process defined by $\bar X_n := (X^{(1)}_n,X^{(2)}_n,\dots)$. Note that $\bar X$ takes values in a discrete countable subset of $(\cA \cup \{\star\})^\N$. Since each $X^{(i)}$ is a finitary factor of $(W,X)$, it is straightforward that $X$ and $\bar X$ are finitarily isomorphic relative to $W$.
For each $i \ge 1$, we apply \cref{thm:relative_smorodinsky_k_dependent} to deduce that $X^{(i)}$ is finitarily isomorphic to \iid\ relatively to $(W,X^{(1)},\dots,X^{(i-1)})$. Together this gives that $X$ is isomorphic relatively to $W$ to an \iid\ process (in which each coordinate consists of a sequence of independent variables), where the isomorphism satisfies the claimed properties above.
\end{proof}

%
%
\section{Entropy-efficient finitary codings for $\Z$-processes}\label{sec:sequential}

In this section, we use \cref{thm:pro-dependent} to deduce the following relative version of \cref{thm:main} for $\Z$-processes:

\begin{thm}\label{thm:realtive_efficient_finitary_coding_Z}
	Let $X$ be a countable-valued $\Z$-process which is a finitary factor of \iid$\times W$. Then for any $\epsilon >0$ there exists a process $X'$ with $H(X'_0) \le \epsilon$ such that $(X,X')$ is finitarily isomorphic to \iid\ relative to $W$. Moreover, when $h(X \mid W) < \infty$, the latter \iid\ process can be taken to be finite-valued, and in particular, there exists a finite-valued \iid\ process $Y$ with $h(Y) \le h(X \mid W) + \epsilon$ such that $X$ is a finitary factor of $Y \times W$.
\end{thm} 

The theorem will follow easily from \cref{thm:pro-dependent} and the following result. Given a finitary factor $X$ of $Y$, we say that an $\N$-valued process $R$ is a \textbf{coding length process for $X$ by $Y$} if $R_0$ is an almost surely finite stopping time with respect to the filtration $\cF=(\cF_n)_{n \ge 0}$ defined by $\cF_n := \sigma(\{Y_i : |i| \le n\})$, and $X_0$ is $\cF_{R_0}$-measurable (recall the definition of stop-finitary). Observe that one can always find a coding length process $R$ with $H(R_0)$ arbitrarily small. Recall the definition of finitarily pro-dependent from before \cref{thm:pro-dependent}.
 
%

\begin{prop}\label{prop:finitary_limit_of_finitarily_dependent}
	Let $X$ be a finitary factor of $(W,Y)$, where $W$ is an aperiodic ergodic process and $Y$ is an \iid\ process independent of $W$. Let $R$ be a coding length process for $X$ by $(W,Y)$. Then $(X,R)$ is finitarily pro-dependent relative to $W$.
\end{prop}

We remark that one cannot hope to prove that $X$ itself is finitarily pro-dependent relative to $W$, since (as shown by Gabor \cite{gabor2019finitary}) there exist finitary factors of \iid\ which are not finitarily isomorphic to \iid\ (and hence not finitarily pro-dependent relative to an independent \iid\ process).
Before giving the proof of the proposition, let us illustrate what could go wrong when trying to show that $(X,R)$ is finitarily pro-dependent relative to $W$. A ``natural candidate'' for  the sequence of partial processes $X^{(n)}$ witnessing finitary pro-dependence could be
\[ X^{(n)}_i := \begin{cases} (X_i,R_i) &\text{if }R_i \le n \\ \star &\text{otherwise} \end{cases} .\]
While this definition yields that each $X^{(n)}$ is a block factor of $(X,R)$, and also even a block factor of \iid, it may fail to satisfy that $X^{(n)}$ is finitarily dependent relative to $(W,X^{(1)},\dots,X^{(n-1)})$. Indeed, $X^{(n)}$ may even fail to be a finitary factor of an \iid\ process relative to $(W,X^{(1)},\dots,X^{(n-1)})$.
For example, let $Y$ be any non-trivial $\{0,1\}$-valued \iid\ process and let $X$ be the $\{0,1,2\}$-valued block factor of $Y$ defined by $X_i=2$ if $Y_i = Y_{i+1}$ and $X_i=Y_{i+2}$ otherwise. Note that $R_i=1$ when $Y_i=Y_{i+1}$ and that $R_i=2$ otherwise. It follows that $Y$ is a 2-to-1 extension of $X^{(1)}$.
In particular, $X^{(2)}$ is not a finitary factor of \iid\ relative to $X^{(1)}$.

\begin{proof}[Proof of \cref{prop:finitary_limit_of_finitarily_dependent}]

Let $M^{(n)}$ be a decreasing sequence of marker processes with density tending to 0, all of which are finitary factors of $W$ (e.g., using \cref{lem:low-density-markers}).
Define $X^{(n)}$ by
\[ X^{(n)}_i := \begin{cases} (X_i,R_i) &\text{if }M^{(n)}_j=0\text{ for all $j$ such that }|j-i| \le R_i\\ \star &\text{otherwise} \end{cases} .\]
Clearly, $X^{(1)},X^{(2)},\ldots$ is a sequence of partial processes increasing to $(X,R)$, and each $X^{(n)}$ is a finitary factor of $(W,X,R)$. It remains to show that each $X^{(n)}$ is finitarily dependent relative to $(W,X^{(1)},\dots,X^{(n-1)})$. In fact, we will show that $X^{(n)}$ is finitarily $M^{(n)}$-block-dependent relative to $(W,X^{(1)},\dots,X^{(n-1)})$.
Note that $X^{(n)}_i = \star$ if and only if $[i-R_i-1,i+R_i] \not\subset I^{M^{(n)}}_i$ (in particular, $X^{(n)}_i=\star$ whenever $M^{(n)}_i=1$).
It follows from this that $X^{(n)}$ is a finitary $M^{(n)}$-block factor of $Y$ relative to $W$ (since $R$ is also a coding length process for $(X,R)$ by $(W,Y)$). In particular, $(X^{(1)},\dots,X^{(n)})$ is a finitary $M^{(n)}$-block factor of \iid\ relative to $W$. It follows that $X^{(n)}$ is finitarily $M^{(n)}$-block-dependent relative to $(W,X^{(1)},\dots,X^{(n-1)})$.
\end{proof}

\begin{proof}[Proof of \cref{thm:realtive_efficient_finitary_coding_Z}]
If $W$ is aperiodic then the theorem follows immediately from \cref{prop:finitary_limit_of_finitarily_dependent} and \cref{thm:pro-dependent} (using a coding length process $R$ with $H(R_0) \le \epsilon$).
If $W$ is periodic, then we take an \iid\ process $V$ of low entropy, independent of everything else, and apply \cref{prop:finitary_limit_of_finitarily_dependent} and \cref{thm:pro-dependent} to deduce that $(X,R)$ is finitarily isomorphic to \iid\ relatively to $(W,V)$. It then follows from this that $(X,R,V)$ is finitarily isomorphic to \iid\ relatively to $W$.
\end{proof}

\begin{cor}\label{cor:relatively_finitary_iid}
	Let $W$ be an aperiodic ergodic process and let $X$ be a countable-valued process.
	Then $X$ is finitarily isomorphic to \iid\ relative to $W$ if and only if there is a finitary factor $X'$ of $(X,W)$ such that $(X,X')$ is finitarily pro-dependent relative to $W$.
\end{cor}
\begin{proof}
If $(X,X')$ is finitarily pro-dependent relative to $W$, then it is finitarily isomorphic to \iid\ relative to $W$ by \cref{thm:pro-dependent}, and if $X'$ is a finitary factor of $(X,W)$, then $X$ and $(X,X')$ are finitarily isomorphic relative to $W$.

Suppose that $X$ is finitarily isomorphic to an \iid\ process $Y$ relative to $W$.
It suffices to show that any coding length process $R$ for $X$ by $(Y,W)$ is a finitary factor of $(X,W)$, since \cref{prop:finitary_limit_of_finitarily_dependent} will then imply that $(X,R)$ is finitarily pro-dependent relative to $W$.
Indeed, it is immediate from the definition that $R$ is a finitary factor of $(Y,W)$, and since $Y$ is a finitary factor of $(X,W)$, we see that $R$ is also a finitary factor of $(X,W)$.
\end{proof}

\section{Nice permutation groups}\label{sec:nice-groups} 
Having established a version of our main result for $\Z$-processes, we now move on to consider the more general $(\V,\Gamma)$-processes, where $\Gamma$ is a group of permutations on a countable set $\V$.

 We say that $\Gamma$ is \textbf{semi-nice} if the following hold:
\begin{enumerate}
	\item $\Gamma$ acts transitively on $\V$.
	\item Stabilizers of $\Gamma$ have finite orbits.
\end{enumerate}
We say that $\Gamma$ is \textbf{nice} if it furthermore satisfies the following:
\begin{enumerate}
	\setcounter{enumi}{2} 
	\item $\Gamma$ is  unimodular.
	\item Every non-trivial \iid\ $(\V,\Gamma)$-process is aperiodic.
\end{enumerate}

Let us explain the notions in the definition. The group $\Gamma$ acts transitively on $\V$ if for any $u,v \in \V$ there exists $g \in \Gamma$ such that $g(u)=v$. By stabilizers having finite orbits, we mean that $|\Gamma_v w|<\infty$ for every $v,w \in \V$, where
\[ \Gamma_v := \{ g \in \Gamma : g(v)=v \} \]
is the \textbf{stabilizer} of $v$. This is equivalent to the condition that the stabilizers have compact closures, by which we mean the closure $\overline\Gamma_v$ of $\Gamma_v$ in $\V^\V$ (with respect to the topology of pointwise convergence, where $\V$ is discrete) is compact for every $v \in \V$. 
From this it follows that condition (2) implies that the closure $\overline{\Gamma}$ of $\Gamma$ in $\V^\V$ is a locally compact group. 
Suppose now that $\Gamma$ is semi-nice.
We say that $\Gamma$ is \textbf{unimodular} if $|\Gamma_v u| = |\Gamma_u v|$ for every $u,v \in \V$. While this definition is given in terms of the action of $\Gamma$ on $\V$, by standard arguments~(see e.g.\ \cite{lyons2017probability}\footnote{The groups $\Gamma$ considered in \cite{lyons2017probability} are automorphisms of a connected locally finite graph, but the arguments there apply to any semi-nice permutation group.}), it is equivalent to the classical condition that there is a Radon measure $m_{\overline{\Gamma}}$ on the Borel subsets of $\overline\Gamma$ that is invariant with respect to multiplication by elements of $\overline{\Gamma}$ from the right and from the left.
In this case, the measure $m_{\overline{\Gamma}}$ is unique up to scaling and is called \textbf{Haar measure} on $\overline{\Gamma}$.
We say that a $(\V,\Gamma)$-process $Y = (Y_v)_{v \in \V}$ is \textbf{aperiodic} if the action of $\Gamma$ on $Y$ is essentially free, meaning that, almost surely, no element of $\Gamma$ other than the identity fixes $Y$.

We say that a semi-nice permutation group $\Gamma$ is \textbf{amenable} if its action on $\V$ is amenable in the sense of Greenleaf~\cite{greenleaf69Amenable}.
This means that there is a $\Gamma$-invariant mean on $\V$, or equivalently, that there exists a F{\o}lner sequence in $\V$ (see \cref{sec:entropy}).
We mention that a semi-nice amenable permutation group need not be amenable as an abstract group, but its closure is amenable as a locally compact group. Moreover, a semi-nice permutation group $\Gamma$ is amenable (in the sense defined above) if and only if $\overline\Gamma$ is amenable as a locally compact group. This follows from the fact that the stabilizers are compact, hence amenable as topological groups. See for instance \cite{glasner2007amenable}.


Motivating examples for nice permutation groups examples are the following:

\begin{itemize}
	\item $\V$ itself is a countable group, and $\Gamma$ is the group of permutations corresponding to multiplication from the left. This is equivalent to saying that the action of $\Gamma$ on $\V$ is transitive and free. The stabilizers are trivial, so that their orbits are singletons and $\Gamma$ is unimodular. Also, $\Gamma$ is discrete (as a subspace of $\V^\V$) so $\Gamma =\overline{\Gamma}$.
	\item Let $G$ be a graph with vertex set $\V$. We say that $G$ has \textbf{uniquely centered balls} if for every two distinct vertices $v$ and $w$ and every $r>0$, the ball of radius $r$ centered around $v$ does not coincide with the ball of radius $r$ centered around $w$. 
	Now if $G$ is a locally finite, connected  vertex-transitive unimodular graph having uniquely centered balls, and $\Gamma$ is the group of automorphism of the graph $G$ (namely, permutations of $\V$ that send edges to edges), then $\Gamma$ is a a nice permutation group. This case has been  considered in \cite{spinka2019finitely}. We remark that the uniquely centered balls assumption can be replaced by slightly weaker conditions  (see, e.g.,  \cite[(24)]{haggstrom2002coupling}). For an amenable graph, the group of automorphisms is amenable, and any  transitive subgroup of the automorphisms is unimodular \cite{salvatori1992,soardi1990amenability}. Thus, when $G$ as above is also amenable, any group of automorphisms that acts transitively on the vertices is a nice amenable group. 
In particular, it is straightforward to deduce the aperiodicity property from the uniquely centered balls condition; see for instance the proofs of Lemma~4.5 in \cite{haggstrom2002coupling} and Lemma~5.1 in \cite{spinka2019finitely}. 
\end{itemize}


\begin{remark}\label{rem:semi_nice_locally_compact}
In various contexts, it is natural to require the group $\Gamma$ to be closed. For our purposes, this is not needed, but can be assumed when convenient due to the following observations:
Every $(\V,\Gamma)$-process is also a $(\V,\overline{\Gamma})$-process (and vice versa), and a (finitary) factor map between $(\V,\Gamma)$-processes is also a (finitary) factor map between $(\V,\overline{\Gamma})$-processes (and vice versa). Moreover, each of the conditions (1) and (2) holds for the group $\Gamma$ if and only if it holds for its closure $\overline{\Gamma}$. Thus, $\Gamma$ is semi-nice if and only if $\overline\Gamma$ is. For semi-nice groups, each of the conditions (3) and (4) holds for $\Gamma$ if and only if it holds for $\overline{\Gamma}$. These observations show that, for most purposes, we can replace $\Gamma$ by $\overline{\Gamma}$.
\end{remark}

\subsection{Aperiodicity and random total orders}

Recall that a $(\V,\Gamma)$-process is said to be aperiodic if the action of $\Gamma$ on it is essentially free. 
For a  $\Z$-process $Y$, aperiodicity means that the probability that there exists an integer $p \ge 1$ such that $Y_{n+p} = Y_n$ for all $n \in \Z$ is zero.
 It will be useful for us to reformulate this in a similar manner as for $\Z$-processes. This is the content of the following lemma, which, informally, says that aperiodicity is equivalent to the ``configuration seen from'' each $v \in \V$ being distinct.

For $u,v \in \V$, denote
\[ \Gamma_{u,v} := \big\{g \in \Gamma : g(u)= v\big\}.\]
For $F \subset \Gamma$ and $y \in \cA^\V$, denote
\[ F Y := \big\{ g(Y) :  g \in F \big\} .\]

\begin{lemma}\label{lem:aperiodic}
Let $\Gamma$ be a semi-nice permutation group and let $Y$ be a $(\V,\Gamma)$-process. Then $Y$ is aperiodic if and only if $\Gamma_{u,v} Y \neq \Gamma_v Y$ almost surely for all $u,v \in \V$ with $u \neq v$.
\end{lemma}
\begin{proof}
Let $\mathsf \Gamma < \Gamma$ be the random subgroup of $\Gamma$ consisting of those $g\in \Gamma$ for which $g(Y)=Y$. Aperiodicity of $Y$ is equivalent to the statement that $\mathsf \Gamma$ is trivial almost surely.
Thus, it suffices to show the deterministic statement that $\Gamma_{u,v} Y \neq \Gamma_v Y$ for all $u \neq v$ if and only if $\mathsf\Gamma$ is trivial. Indeed, since $\Gamma_{g(v),v} = \Gamma_v g^{-1}$ for any $v\in\V$ and $g\in\Gamma$, the former holds if and only if $g^{-1}(Y) \notin \Gamma_v Y$ for all $v \in \V$ and $g \notin \Gamma_v$, which holds if and only if $\mathsf\Gamma \cap g \Gamma_v =\emptyset$ for all $v \in \V$ and $g \notin \Gamma_v$, which in turn is equivalent to $\mathsf\Gamma$ being trivial.
\end{proof}

Let us discuss condition (4) in the definition of a nice permutation group, namely, that every non-trivial \iid\ process is aperiodic.
Our only use of this assumption will be through the following result which shows that any aperiodic $(\V,\Gamma)$-process admits a total order on $\V$ as a finitary factor. The constructed total order will be of the following particular form: For a $\{0,1\}^\N$-valued process $Z$, we denote by $\prec_Z$ the partial order on $\V$ induced by the lexicographical order on $\{0,1\}^\N$, i.e., $u \prec_Z v$ if and only if $Z_u$ is lexicographically smaller than $Z_v$. Clearly, $\prec_Z$ is a total order if and only if $Z_u \neq Z_v$ for all $u \neq v$. We call any total order of this type a \textbf{bitwise total order}.
We say that a process $Y$ admits a \textbf{finitary bitwise total order} if there exists a bitwise total order $\prec_Z$ for a process $Z$ such that $(Z_v(n))_{v \in \V}$ is a finitary $\Gamma$-factor of $Y$ for every $n \in \N$ 
(when $Y$ takes values in a discrete countable set, this is the same as saying that $Z$ is a topo-finitary factor of $Y$). Note, in particular, that in this case the map sending $Y$ to $\prec_Z$ is finitary in the sense that for every $v,w \in \V$ almost surely there exists a (random) finite $F \subset \V$ such that $Y_F$ determines the event ${\{v \prec_Z w\}}$.

\begin{lemma}\label{lem:inv_total_order}
Let $\Gamma$ be a semi-nice permutation group.
Then any finite-valued aperiodic $(\V,\Gamma)$-process $Y$ admits a finitary bitwise total order on $\V$.
\end{lemma}

In light of \cref{lem:aperiodic}, it is easy to see that the converse of \cref{lem:inv_total_order} holds in a strong sense, namely, that a process which admits a bitwise total order as a factor (finitary or otherwise) must be aperiodic.

\begin{proof}
Suppose that $Y$ takes values in a finite set $\cA$.
The space $\cA^{\V}$ is a totally disconnected, compact metrizable space having no isolated points  (with respect to the product topology). The  space $C(\cA^{\V})$ of compact subsets of $\cA^{\V}$  equipped with the Fell topology is again a totally disconnected, compact metrizable space having no isolated points. Hence, it is topologically a Cantor space, i.e., it is homeomorphic to $\{0,1\}^\N$.  Let $\phi:C(\cA^{\V}) \to \{0,1\}^\N$ be a homeomorphism. Using \cref{rem:semi_nice_locally_compact}, we may assume that $\Gamma$ is closed, so that its stabilizers are compact.

Fix $v_0 \in \V$ and define a process $Z$ by
\[ Z_v := \phi(\Gamma_{v,v_0}Y) .\]
It can be directly verified that the map $Y \mapsto Z$ is $\Gamma$-equivariant. Furthermore, $Z$ is a topo-finitary $\Gamma$-factor of $Y$ due to the continuity of $\phi$ and the assumption that the stabilizers have finite orbits. The aperiodicity of $Y$, together with \cref{lem:aperiodic}, imply that map $v \mapsto \Gamma_{v,v_0}Y$ is almost surely injective.
Using that $\phi$ is injective, we obtain that, almost surely, $Z_v \neq Z_u$ for distinct $u$ and $v$, and hence $\prec_Z$ is a total order on $\V$.
%
%
%
\end{proof}


\begin{remark}\label{rem:non_aperiodic_example}
	To further demonstrate the importance of the aperiodicity condition (4) in the definition of a nice permutation group, consider the following example:
 Let $\V=\Z \times \{-1,1\}$ and let $\Gamma$ be the group of permutations generated by the natural shift $(n,i) \mapsto (n+1,i)$ and all (uncountably many) transformations the form $(n,i) \mapsto (n, i s_i)$ where $s \in \{-1,1\}^\Z$. Then $\Gamma$ satisfies all the required properties except for the aperiodicity condition (i.e., it is a semi-nice unimodular permutation group; it is also amenable). Indeed, no countable-valued \iid\ process $Y$ on $\V$ is $\Gamma$-aperiodic since there is positive probability that the element $g \in \Gamma$ that swaps $(0,1)$ with $(0,-1)$ and fixes all other elements of $\V$ stabilizes $Y$.
For this reason, in order for an \iid\ process corresponding to a probability vector $p=(p_i)$ to be a factor of an \iid\ process corresponding to a probability vector $q$, it is necessary that $\sum_{i}p_i^2 \ge \sum_{i}q_i^2$. For example, the \iid\ process corresponding to $(\frac{1}{2},\frac{1}{2})$ is not a $\Gamma$-factor of the (higher entropy) \iid\ process corresponding to $(\frac34,\frac18,\frac18)$. In particular, \cref{thm:iid-isomorphic} (and similarly also \cref{thm:main}) fails in this case.
Actually, for this particular group $\Gamma$, it can be shown that two countable-valued \iid\ processes are finitarily isomorphic if and only if they have equal entropy and  $\sum_{i}p_i^2=\sum_{i}q_i^2$, where $p$ and $q$ are the corresponding probability vectors. This can be proved by  applications of the finitary isomorphism result for $\mathbb{Z}$-processes. We omit the details. 
\end{remark}

\subsection{Entropy for $(\V,\Gamma)$-processes}\label{sec:entropy}

In this section, we define the entropy of a $(\V,\Gamma)$-process when $\Gamma$ is a semi-nice unimodular amenable permutation group, and establish some basic facts about it, analogous to those in the classical setting of ordinary $\Gamma$-processes. Beyond the definition of  entropy of a $(\V,\Gamma)$-process, the results in the this section are not used in other parts of the paper.

Let $\Gamma$ be a semi-nice unimodular amenable permutation group of $\V$ and let $X$ be a $(\V,\Gamma)$-process such that $H(X_v)$, the Shannon entropy of $X_v$ for $v \in \V$, is finite (in particular, $X$ essentially takes values in a countable set).
The \textbf{entropy} of the $(\V,\Gamma)$-process $X$ is given by
\[ h(X) := \inf_{\substack{V \subset \V\text{ finite}\\\text{and non-empty}}} \frac{H(X_V)}{|V|} .\]
In the case where the action of $\Gamma$ on $\V$ is free, this is the classical notion of mean-entropy for a process over a countable amenable group.
As in the classical case, entropy of $(\V,\Gamma)$-processes is monotone under factors:
\begin{prop}\label{prop:monotone-entropy}
	Let $\Gamma$ be a semi-nice unimodular amenable permutation group of $\V$.
	Let $X$ and $Y$ be $(\V,\Gamma)$-processes such that $H(X_v),H(Y_v) < \infty$ for $v \in \V$. If $X$ is a $\Gamma$-factor of $Y$ then $h(X)\le h(Y)$.
\end{prop}

In particular, this shows that the entropy is invariant under measure-theoretic isomorphism for $(\V,\Gamma)$-processes such that $H(X_{v_0})<\infty$, over semi-nice unimodular amenable groups. 
For a $(\V,\Gamma)$-process $X$ such that $H(X_{v_0})=\infty$, with $\Gamma$ a semi-nice unimodular amenable permutation group, we can consistently define
\[ h(X) := \sup\{ h(X') : X' \mbox{ is a factor of } X \mbox{ with } H(X'_v) < \infty \mbox{ for all  } v\in \V\} .\]


\medbreak

We proceed to prove \cref{prop:monotone-entropy}, whose proof we provide for completeness, and to reassure that this result carries over from the classical setting of ordinary $\Gamma$-processes to the setting of $(\V,\Gamma)$-processes. For this, we shall show that $h(X)$ can be computed along any co-F{\o}lner sequence in $\V$ (defined below).

Recall that a sequence of compact subsets $(F_n)_{n=1}^\infty$ in a locally compact  group $\Gamma$ is called a \textbf{bi-F{\o}lner sequence} if 
\[\lim_{n \to \infty} \frac{m_\Gamma(KF_nK \setminus F_n)}{m_\Gamma(F_n)} = 0 \qquad\text{for every compact }K \subset \Gamma .\]
A locally compact Polish group admits a bi-F{\o}lner sequence if and only if it is amenable and unimodular. 
See \cite[``variants of  F{\o}lner's condition'']{ornsteinWeissEntropy1987}.


Getting back to our case of interest, $\Gamma$ is a semi-nice group of permutations of $\V$. In particular, the closure of $\Gamma$ is a locally compact Polish
group. A sequence $(V_n)_{n=1}^\infty$ of finite subsets of $\V$ is called a \textbf{F{\o}lner sequence}  (with respect to the action of  $\Gamma$) if 
\[\lim_{n \to \infty}\frac{| V_n \setminus gV_n|}{|V_n|} = 0 \qquad\text{for any }g \in \Gamma.\]
We call $(V_n)_{n=1}^\infty$ a \textbf{co-F{\o}lner sequence}  (with respect to the action of  $\Gamma$) if
\[\lim_{n \to \infty}\frac{|\partial_{v_1}^{v_0} V_n|}{|V_n|} = \lim_{n \to \infty}\frac{|\partial_{v_1}^{v_0} V_n^c|}{|V_n|} = 0 \qquad \text{for every }v_0,v_1 \in \V ,\]
where
\[
\partial_{v_1}^{v_0} V := \left\{g \in \Gamma : g(v_0) \in V,~ g(v_1) \notin V\right\} v_0.
\]
Since $\partial_{h(v_1)}^{h(v_0)}V=\partial_{v_1}^{v_0}V$ for any $h \in \Gamma$, we may always fix one of $v_0$ or $v_1$ in the definition of a co-F{\o}lner sequence.
Since $|\partial^{v_0}_{v_1} V^c| \le | \Gamma_{v_1} v_0 | \cdot | \partial_{v_0}^{v_1} V |$, if one of the limits in the definition of a co-F{\o}lner sequence is zero (for all $v_1$), then so is the other.
If $(V_n)_{n=1}^\infty$ is both F{\o}lner and co-F{\o}lner, we say it is a \textbf{bi-F{\o}lner sequence}. Let us mention that in the special case where the action of $\Gamma$ on $\V$ is transitive and free, we can identify $\V$ with $\Gamma$. In this case, F{\o}lner and co-F{\o}lner sequences in $\V$ coincide with left and right F{\o}lner sequences in $\Gamma$, respectively,   and then the two notions of bi-F{\o}lner sequences coincide.
Any continuous action of a locally compact amenable group on a locally compact space is an amenable action in the sense of Greenleaf~\cite{greenleaf69Amenable}, i.e., there exists a F{\o}lner sequence in $\V$. Conversely, if a group acts transitively on a locally compact space, the action is amenable and the stabilizer of any point is a compact group, then the acting group is an amenable group. It follows that a semi-nice permutation group $\Gamma$ is amenable (as a locally compact group) if and only if its action on $\V$ is amenable. It turns out that co-F{\o}lner sequences come up naturally in the context of entropy of $(\V,\Gamma)$-processes. The following lemma shows that bi-F{\o}lner sequences (and, in particular, co-F{\o}lner sequences) in $\V$ also exist in this case.

\begin{lemma}\label{lem:biFolner}
	Let $\Gamma$ be a semi-nice unimodular amenable permutation group of $\V$ and let $(F_n)_{n=1}^\infty$ be a bi-F{\o}lner sequence in $\overline{\Gamma}$ such that $\overline{\Gamma}_{v_0}F_n \overline{\Gamma}_{v_0}= F_n$ for all $n$. Let $V_n = F_n v_0$. Then $(V_n)_{n=1}^\infty$ is a bi-F{\o}lner sequence in $\V$, and each $V_n$ is $\Gamma_{v_0}$-invariant.
\end{lemma}

Note that a sequence $(F_n)_{n=1}^\infty$ as in the lemma indeed exists.

\begin{proof}
	Let $v_1 \in \V$ and choose any $g \in \Gamma_{v_0,v_1}$.
	Then (using the fact that $F_n \overline{\Gamma}_{v_0}= F_n$)
	\[
	\partial_{v_1}^{v_0} V_n=\left( F_n \setminus F_n g^{-1}\right) v_0.
	\]
	So (using the fact that $\overline{\Gamma}_{v_0}F_n = F_n$)
	\[|\partial_{v_1}^{v_0} V_n| = m_{\overline{\Gamma}}\left(F_n \setminus F_n g^{-1} \right). \]
	Since $(F_n)$ is a bi-F{\o}lner sequence, we have that $m_{\overline{\Gamma}}\left(F_n \setminus F_n g^{-1} \right) = o(m_{\overline{\Gamma}}(F_n))$ as $n \to \infty$.
	Using that $|V_n| = m_{\overline{\Gamma}}(F_n \Gamma_{v_0}) = m_{\overline{\Gamma}}(F_n)$, we deduce that $(V_n)$ is co-F{\o}lner.

	We proceed to show that $(V_n)$ is a F{\o}lner sequence. Let $g \in \Gamma$. Then
	\[\partial_g V_n = F_nv_0 \setminus gF_nv_0 \subseteq (F_n \setminus gF_n)v_0 \]
	Using the fact that $F_n \overline{\Gamma}_{v_0}= F_n$, it follows that $|\partial_g V_n| \le m_{\overline{\Gamma}}(F_n \setminus gF_n)$, and since $(F_n)$ is a bi-F{\o}lner, it follows that $(V_n)$ is a F{\o}lner sequence.
\end{proof}

Thus, if $\Gamma$ is a semi-nice unimodular amenable permutation group of $\V$, then there exists a bi-F{\o}lner sequence $(V_n)_{n=1}^\infty$ in $\V$, with each $V_n$ being $\Gamma_{v_0}$-invariant.
In fact, a minor adaptation of the proof in the setting of graphs given in \cite[Section~8]{lyons2017probability} shows that a semi-nice (closed) permutation group that admits a co-F{\o}lner sequence in $\V$ is unimodular and amenable (and hence admits a bi-F{\o}lner sequence).

As in the classical setting, the entropy of a process can actually be computed as a limit along a co-F{\o}lner sequence (see \cref{cor:h_inf_formula} below). This is a simple consequence of the following inequality. 
For $W,U \subset \V$, denote
\[
\partial_{U}^{v_0}W := \bigcup_{u \in U}\partial_u^{v_0}W.
\]

\begin{lemma}\label{lem:H_almost_subaddtivity1}
	Let $\Gamma$ be a semi-nice permutation group of $\V$ and let $W,U \subset \V$ be finite and $u_0 \in U$. Then
	\[
	\frac{H(X_{W})}{|W|} \le \frac{H(X_{U})}{|U|} \cdot \left(1+ \frac{|\partial^{u_0}_U W^c|}{|W|} \right).
	\]
\end{lemma}

\begin{proof}
	Observe that $\Gamma_U := \bigcap_{u \in U} \Gamma_u$ is a subgroup of $\Gamma_{u_0}$ of finite index $\ell := [\Gamma_U : \Gamma_{u_0}]$. Let $\Gamma' := \Gamma / \Gamma_U$, and note that $gu$ is well defined for $g \in \Gamma'$ and $u \in U$, in the sense that it does not depend on the representative. Consider the (multi-)collection of sets
	\[ \cK := \{ gU : g \in \Gamma',~ gU \cap W \neq \emptyset \} .\]
	We claim this is a $\ell|U|$-cover of $W$, meaning that each $w \in W$ is contained in exactly $\ell|U|$ many sets in $\cK$. Indeed, for each $u \in U$, there are $\ell$ elements $g \in \Gamma'$ such that $gu=w$ (and these elements are clearly distinct for different $u$). Thus, Shearer's inequality (and using the $\Gamma$-invariance of $X$) yields that
	\[ H(X_W) \le \frac {H(X_U)}{\ell|U|} \cdot |\cK| .\]
	It remains only to show that $|\cK| \le \ell (|W| + |\partial_U^{u_0} W^c|)$. Indeed, this follows since
	\[ |\cK| = |\{ g \in \Gamma' : gU \cap W \neq \emptyset \}| = \ell \cdot |\{ g \in \Gamma : gU \cap W \neq \emptyset \}u_0| \]
	and
	\[ \{ g \in \Gamma : gU \cap W \neq \emptyset \}u_0 \\
	= W \cup \{ g \in \Gamma : gU \cap W \neq \emptyset,~ g(u_0) \notin W \}u_0
	= W \cup \partial^{u_0}_U W^c. \qedhere \]

\end{proof}

From \cref{lem:H_almost_subaddtivity1} and the definition of a co-F{\o}lner sequence, we immediately get:

\begin{cor}\label{cor:h_inf_formula}
Let $\Gamma$ be a semi-nice permutation group of $\V$ and let $X$ be a $(\V,\Gamma)$-process such that $H(X_v)<\infty$.
Then for any co-F{\o}lner sequence $(V_n)_{n=1}^\infty$ in $\V$, we have
\[ h(X) = \lim_{n \to  \infty}\frac{H(X_{V_n})}{|V_n|} .\]
\end{cor}

We are now ready to prove \cref{prop:monotone-entropy}.

\begin{proof}[Proof of \cref{prop:monotone-entropy}]
	Let $X$ be a factor of $Y$ and fix $\epsilon>0$.
	There exists a finite set $F \subset \V$ so that $H(X_{v_0} \mid Y_F) < \epsilon$.
	Let $(V_n)_{n \in \N}$ be a co-F{\o}lner sequence in $\V$, which exists by \cref{lem:biFolner}.  
	If $v \in  V_{n}\setminus \partial_F^{v_0}V_n$ and $g \in \Gamma_{v_0,v}$ then $F \subseteq g^{-1}(V_n)$. Thus,
	\[H(X_v \mid Y_{V_n}) = H(X_{v_0} \mid Y_{g^{-1}(V_n)}) \le H(X_{v_0} \mid Y_{F}) \le \epsilon.\]
	Hence,
	\[ H(X_{V_n}\mid Y_{V_n}) \le H(X_{\partial_F^{v_0} V_n}) + H(X_{V_{n}\setminus \partial_F^{v_0}V_n} \mid Y_{V_n}) \le |\partial_F^{v_0} V_n|H(X_{v_0}) + \epsilon |V_n|.\]
	Since $(V_n)$ is co-F{\o}lner and $\epsilon$ was arbitrarily, we conclude that $\frac{1}{|V_n|}H(X_{V_n} \mid Y_{V_n}) \to 0$ as $n \to \infty$. Finally, using \cref{cor:h_inf_formula}, $h(X)=\lim_{n\to\infty} \frac{1}{|V_n|}H(X_{V_n}) \le \lim_{n\to\infty} H(Y_{V_n}) = h(Y)$.
\end{proof}

We conclude this section with a few additional remarks, starting with the roles of amenability and unimodularity in the entropy theory of $(\V,\Gamma)$-processes. A famous example of Ornstein and Weiss~\cite{ornstein1980amenablerohlin} shows that over a free group, an \iid\ process can have \iid\ factors of greater entropy. In fact, on a regular tree, any \iid\ process is an automorphism-equivariant factor of the uniform 2-valued  \iid\ process \cite{ball2005factors}. In particular, the amenability assumption in \cref{prop:monotone-entropy} cannot be dropped.
The unimodularity assmption also cannot be dropped. To see this, consider the automorphism group $\Gamma$ of the grandparent graph \cite[Example 7.1]{lyons2017probability}, or equivalently, the group of automorphisms of a regular tree which fix a given end, viewed as a permutation group of the vertex set $\V$. It is well known that this group is amenable as a locally compact group \cite{cartwright1994random} (equivalently, the action is amenable in the sense of Greenleaf), but not unimodular. For references see \cite[Section 7.9]{lyons2017probability}. Since $\Gamma$, as a group of permutations of $\V$, is a subgroup of the full automorphism group of a regular tree, it follows that in this case too, an \iid\ process can have \iid\ factors of greater entropy. We note that the latter also gives an example of a semi-nice amenable permutation group which is non-unimodular and admits no co-F{\o}lner sequence (for semi-nice amenable permutation groups, the latter two properties are equivalent).

Ornstein and Weiss \cite{ornsteinWeissEntropy1987} introduced an invariant which generalizes classical Kolmogorov--Sinai entropy for a very general class of \emph{essentially free} actions of locally compact unimodular amenable groups (under a certain mild condition which applies in great generality).  It is of interest to explore the precise relation between the entropy of a $(\V,\Gamma)$-process and the Ornstein--Weiss entropy of the associated $\overline{\Gamma}$-action. However, this would be a detour and we do not pursue this issue here.

\section{Finitary $\Z$-type orders}\label{sec:z_type_order}
	A classical theorem of Ornstein and Weiss~\cite{ornstein1980amenablerohlin} states that all essentially free, ergodic actions of a countable group are hyperfinite, and in particular, orbit equivalent to a $\Z$-action. The latter part can be reinterpreted as follows:
	For any countable group $\Gamma$, any essentially free and ergodic $\Gamma$-process $X$  admits a factor that is a random invariant $\Z$-type total order. See \cite{downarowicz2021multiorders} for more on this point of view. The main result of this section, \cref{prop:cycle-free-permutation}, is a finitary version of an analogous statement regarding \iid\ $(\V,\Gamma)$-processes for nice amenable permutation groups. Note that the $\Gamma$-orbits of a $(\V,\Gamma)$-process can be uncountable, thus clearly not orbit equivalent to a $\Z$-action in the classical sense.

Let $\Gamma < \mathit{Perm}(\V)$ be a group.
For $z \in \V^\V$ and $g \in \Gamma$, we define $T^g(z)_v = g(z_{g^{-1}(v)})$. This defines an  action of $\Gamma$ on $\V^\V$. 
We say that a process $Y$ admits a \textbf{cycle-free permutation} of $\V$ as a finitary factor  if there exists a topo-finitary map 
$\pi$ from $Y$ to $\mathit{Perm}(\V)$ such that for any $g \in \Gamma$, almost surely, $T^g(\pi(Y))=\pi(g(Y))$, and, almost surely, $\pi(Y)$ is a permutation of $\V$ with no finite orbits. Note that if a permutation of $\V$ has a single orbit, then it can be seen as a $\Z$-type order on $\V$.

\begin{prop}[Finitary cycle-free permutation and $\Z$-type order]\label{prop:cycle-free-permutation}
Let $\Gamma < \mathit{Perm}(\V)$ be a nice  permutation  group and let $Y$ be a non-trivial \iid\ process. Then $Y$ admits a cycle-free permutation of $\V$ as a finitary factor. Furthermore, if $\Gamma$ is amenable and $Y$ has at least three symbols, then the permutation can be chosen to have a single orbit almost surely.
\end{prop}

The proof of \cref{prop:cycle-free-permutation} is based on constructing a certain increasing sequence of partitions of $\V$, captured by the following notion.
We say that an $\cA$-valued process $Y$ admits a \textbf{finitary cell process} if it admits a sequence $\rho=(\rho_j)$ of finitary maps $\rho_j:\cA^\V \to \V^\V$ such that for every $j \in \N$, $g \in \Gamma$ and $v,w \in \V$, almost surely:
\begin{enumerate}[(i)]
\item $\rho_j(g(Y))=T^g(\rho_j(Y))$.
\item If $\rho_j(Y)_v = \rho_j(Y)_w$ then $\rho_{j+1}(Y)_v =  \rho_{j+1}(Y)_w$.
\item $\{u \in \V : \rho_j(Y)_u = v\} \subset V_{j,v}$ for some finite deterministic set $V_{j,v} \subset \V$.\footnote{Other reasonable conditions are also possible. For example, we could have made do with the weaker requirement that the cell $\{u \in \V : \rho_j(Y)_u = v\}$ is a finitary function of $Y$.} 
\end{enumerate}


Note that for any $j \in \N$, the function $\rho_j(Y) \colon \V \to \V$ induces a partition of $\V$ with finite ``cells''.
The partition induced by $\rho_{j+1}(Y)$ is coarser than the partition induced by $\rho_j(Y)$. Thus, we obtain another partition of $\V$ in the limit as $j \to \infty$. We call this the \textbf{eventual partition}. We stress that the eventual partition is not necessarily finitary in any sense: since the partitions becomes coarser as $j$ increases, there will be a finite witness for the event that $v$ and $w$ are in the same eventual partition class, but there need not be such a witness for the complement of this event.

We say that a cell process has \textbf{infinite eventual cells} if, almost surely, the eventual partition has no finite partition classes, or equivalently, if for every $v \in \V$ we have that $|\{u \in \V : \rho_j(Y)_u = \rho_j(Y)_v\}| \to \infty$ as $j \to \infty$. We say that a cell process has a \textbf{single eventual cell} if, almost surely, the eventual partition is $\{\V\}$, or equivalently, if for every $v,w \in \V$ there exists $j \in \N$ such that $\rho_j(Y)_v=\rho_j(Y)_w$. Clearly, if a cell process has a single eventual cell, then it also has infinite eventual cells.

\begin{prop}[Finitary cell process] \label{prop:finite_cell_process}
	Let $\Gamma < \mathit{Perm}(\V)$ be a nice permutation group and let $Y$ be a non-trivial \iid\ process. Then $Y$ admits a finitary cell process having infinite eventual cells. Furthermore, if $\Gamma$ is amenable and $Y$ has at least three symbols, then the cell process can be chosen to have a single eventual cell.
\end{prop}


\subsection{Finitary cell process}

In this section, we give the main constructions towards establishing \cref{prop:finite_cell_process} about the existence of finitary cell processes.
Our construction in the amenable case (leading to a single eventual cell) is different than the one in the general case. We state these in two separate lemmas. In the amenable case, we state a weaker result, which we then use in \cref{sec:concluding-cell-process} to conclude the full strength of \cref{prop:finite_cell_process}.

\begin{lemma}\label{lem:finite_cell_processI}
	Let $\Gamma < \mathit{Perm}(\V)$ be a semi-nice permutation group. Then any finite-valued aperiodic $(\V,\Gamma)$-process admits a finitary cell process having infinite eventual cells.
\end{lemma}

\begin{proof}
	Fix $v_0 \in \V$ and let 
		$V_n$  be a sequence of finite $\Gamma_{v_0}$-invariant sets increasing to $\V$. Such a sequence can be obtained by taking  $V_n=\Gamma_{v_0}V_n'$ where $V_n'$ is any  sequence of finite subsets increasing to $\V$. Denote $V_n(v) := \Gamma_{v_0,v} V_n$. We note that $\{ u \in \V : v \in V_n(u) \}$ is finite for any $v \in \V$.
	
	Let $Y$ be a finite-valued aperiodic $(\V,\Gamma)$-process.
	By \cref{lem:inv_total_order}, $Y$ admits a finitary bitwise total order on $\V$. That is, there exists a $\{0,1\}^\N$-valued process $Z$ such that $\prec_Z$ is a total order on $\V$ and each $Z^n$, the pointwise restriction of $Z$ to $\{1,\dots,n\}$, is a finitary $\Gamma$-factor of $Y$.
	For each $n \ge 1$, let $\prec_n$ be the partial order on $\V$ induced by the lexicographical order on $\{0,1\}^n$ given by $Z^n$. Then each $\prec_n$ is a finitary factor of $Y$ and the sequence $(\prec_n)$ increases to $\prec$.
	
	For $v \in \V$, let $N_v$ be the smallest $n \ge 1$ such that $V_n(v)$ contains an element $\prec$-smaller than $v$. Note that $N_v$ is almost surely finite since there is no $\prec$-minimal element almost surely. Let $v^-$ denote the $\prec$-smallest element in $V_{N_v}(v)$ (this is well defined since $N_v$ is almost surely finite). Let $D$ be the directed graph on $\V$ whose edges are $(v,v^-)$ for all $v \in \V$. Note that every vertex has out-degree 1 in $D$. 
	In particular, for any starting vertex $v \in \V$, the sequence $(v,v^-,v^{--},\dots)$ is well defined and is an infinite
	forward-directed path of distinct vertices in $D$.

	
	For each $n \ge 1$, let $D_n$ be the directed subgraph of $D$ whose vertex set is $\V$ and where the edge $(v,v^-)$ belongs to $D_n$ if and only if $v^- \in V_n(v)$ and $v^- \prec_n v$. Since $V_n(v)$ increases to $\V$ and $\prec_n$ increases to $\prec $, we have that $D_n$ increases to $D$ as $n \to \infty$. Let $D^*_n$ be the non-directed graph underlying $D_n$. 
	We claim that, almost surely, all connected components in $D^*_n$ are finite. Since the out-degrees in $D_n$ are at most 1 and the in-degrees are bounded, this is equivalent to the statement that every (forward or backward) directed path in $D_n$ is finite. Indeed, since $Z^n_v$ takes at most $2^n$ values, any directed path in $D_n$, which in particular constitutes a strictly $\prec_n$-monotone sequence, must be finite (in fact, has length at most $2^n$). Furthermore, the connected component of each vertex $v$ in $D^*_n$ is a finitary function of $Y$.


	Each connected competent of $D^*_n$ is a tree, and since the out-degrees in $D_n$ are at most 1, there is a unique element in each connected competent of $D^*_n$ whose out-degree in $D_n$ is zero.
	The finitary cell process is now obtained by setting $\rho_n(y)_v$ to be the unique element in the component of $v$ in $D^*_n$ whose out-degree in $D_n$ is zero. It is straightforward to check that the requirements of a finitary cell process are satisfied.
\end{proof}

\begin{lemma}\label{lem:finite_cell_processII}
	Let $\Gamma < \mathit{Perm}(\V)$ be a semi-nice unimodular amenable permutation group. Then for any $\epsilon>0$ there exists a (countable-valued) \iid\ process with entropy at most $\epsilon$ which admits a finitary cell process having a single eventual cell.
\end{lemma}

\begin{proof}

For convenience, using \cref{rem:semi_nice_locally_compact}, we replace $\Gamma$ with its closure $\overline{\Gamma}$,  so that $\Gamma_{v_0}$ is compact. 
Let $(F_j)_{j=0}^\infty$ be an increasing bi-F{\o}lner sequence in $\Gamma$ such that $\Gamma=\bigcup_{j=0}^\infty F_j$ and with compact $F_j$'s. By replacing $F_j$ with $\Gamma_{v_0}F_j\Gamma_{v_0}$ we can assume without loss of generality that each $F_j$ is invariant under multiplication by $\Gamma_{v_0}$ from the left and from the right (here we use that $\Gamma_{v_0}$ is a compact subgroup of $\Gamma$). We may similarly also assume that $F_j =F_j^{-1}$ for all $j$ and that $F_0=\{1_\Gamma\}$. Under these assumptions, normalizing the Haar measure so that $m_\Gamma(\Gamma_{v_0})=1$, the Haar measure of each $F_j$ is equal to the cardinality of $F_j v_0$.

We define an increasing sequence of integers $(n_j)_{j=0}^\infty$ by induction. We start by setting $n_0:=0$. Let $j \ge 1$ and suppose that $n_{j-1}$ has been defined. Denote
\[ K_j := (F_{n_{j-1}} \cdots F_{n_2} F_{n_1}) F_j (F_{n_1}F_{n_2}\cdots F_{n_{j-1}}). \]
Using that $(F_{n})_{n=1}^\infty$ is a F{\o}lner sequence, choose $n_j$ large enough so that $N_j := m_\Gamma(F_{n_j})$ satisfies
 \begin{equation}\label{eq:cell-process-folner}
 m_\Gamma\left(K_j F_{n_j}\triangle F_{n_j} \right)< \tfrac14 N_j
 \end{equation}
 and
 \[ H(1/N_j) < \epsilon \cdot 2^{-j} .\]

Let $(W^j)_{j=1}^\infty$ be a sequence of processes which are mutually independent of one another and with $W^j$ being an \iid\ percolation process of density $1/N_j$. Let $Y$ be the process defined by $Y_v := (W^j_v)_{j=1}^\infty$. Note that $h(Y)=\sum_{j=1}^\infty h(W^j)<\epsilon$ and that, in particular, $Y$ is a countable-valued \iid\ process. We will construct the cell process as a finitary factor of $Y$.

We are now ready to define a cell process $\rho = (\rho_j)_{j=1}^\infty$.
Each $\rho_j$ will be a finitary map from $(\{0,1\}^\N)^\V$ to $\V^\V$.
In fact, to define $\rho_j$, we will only use the finite sequence $(W^1,\dots,W^j)$ rather than the entire sequence $(W^1,W^2,\dots)$, so that $\rho_j$ can be thought of as a finitary map from $(\{0,1\}^j)^\V$ to $\V^\V$.
We will define $\rho_j$ by induction.

 For convenience, denote
\[ Z^j := \rho_j(Y) = \rho_j(W^1,\dots,W^j) .\]
We begin by setting
\[ Z^0_v := v \qquad\text{for all }v \in \V .\]
Now let $j \ge 1$ and suppose we have already defined $Z^{j-1}$.
Define
\[ Z^j_v := \begin{cases}
u & \mbox{if } C_j(Z^{j-1}_v) = \{u\} \\
Z^{j-1}_v & \mbox{if } |C_j(Z^{j-1}_v)| \neq 1
\end{cases}
,\]
where
\[ C_j(v) := W^j \cap \Gamma_{v_0,v} F_{n_j} v_0 .\]
This completes the definition of the cell process.

We now show that the desired properties are satisfied.
First, we claim that for every $j \in \N$ and $v,w \in \V$:
\begin{enumerate}
\item $T^g(\rho_j(\omega))=\rho_j(g(\omega))$ for all $g \in \Gamma$ and $\omega \in (\{0,1\}^j)^\V$.
\item If $Z^j_v =  Z^j_w$ then $ Z^{j+1}_v =  Z^{j+1}_w$.
\item $Z^j_v$ belongs to $\Gamma_{v_0,v} F_{n_1}F_{n_2}\cdots F_{n_j} v_0$ and depends only on the values of $Y$ on this set.
\item If $Z^j_w = v$ then $w \in \Gamma_{v_0,v} F_{n_j} \cdots F_{n_2}F_{n_1} v_0$.
\end{enumerate}
These properties are easily verified by induction: for (1) we use that $g(C_j(v)) = g(W^j) \cap \Gamma_{v_0,g(v)} F_{n_j} v_0$, for (3) we use that $w \in \Gamma_{v_0,v} F v_0$ implies $\Gamma_{v_0,w} \subset \Gamma_{v_0,v} F \Gamma_{v_0}$, and for (4) we also use that $w \in \Gamma_{v_0,v} F v_0$ if and only if $v \in \Gamma_{v_0,w}F^{-1}v_0$. Note that $K v_0$ is a finite subset of $\V$ for any compact subset $K$ of $\Gamma$ (since it is the image of a compact set under the continuous map $g \mapsto g(v_0)$). This shows that $\rho$ is a finitary cell process.

It remains to show that the cell process has a single eventual cell.
Fix $v,w \in \V$ and let $j \in \N$ be large enough so that $v \in \Gamma_{v_0,w} F_j w$. Denote $\tilde v := Z^{j-1}_v$ and $\tilde w := Z^{j-1}_w$.
Define
\[ \tilde N_j := | \Gamma_{v_0,\tilde v}F_{n_j} v_0 \cap \Gamma_{v_0,\tilde w}F_{n_j}v_0| \qquad\text{and}\qquad \tilde M_j := | \Gamma_{v_0,\tilde v}F_{n_j} v_0  \triangle \Gamma_{v_0,\tilde w}F_{n_j}v_0| .\]
Let $g \in \Gamma_{\tilde w, v_0} \Gamma_{v_0,\tilde v}$ and note that
\[ \tilde N_j = m_{\Gamma}\left(g F_{n_j} \cap F_{n_j} \right) \le N_j \qquad\text{and}\qquad  \tilde M_j = m_\Gamma\left(g F_{n_j}\triangle F_{n_j} \right) \le 2N_j .\]
To get a lower bound on $\tilde N_j$, note that $\tilde v \in \Gamma_{v_0,v} F_{n_1}\cdots F_{n_{j-1}} v_0$ and $\tilde w \in \Gamma_{v_0,w} F_{n_1}\cdots F_{n_{j-1}} v_0$ imply that $\Gamma_{v_0,\tilde v} \subset \Gamma_{v_0,v} F_{n_1}\cdots F_{n_{j-1}}$ and $\Gamma_{v_0,\tilde w} \subset \Gamma_{v_0,w} F_{n_1}\cdots F_{n_{j-1}}$, and that $v \in \Gamma_{v_0,w} F_j w$ implies that $\Gamma_{v_0,v} \subset \Gamma_{v_0,w} F_j$, so that
\[ g \in \Gamma_{\tilde w, v_0} \Gamma_{v_0,\tilde v} \subset (F_{n_{j-1}} \cdots F_{n_1}) \Gamma_{w,v_0} \Gamma_{v_0,v} (F_{n_1}\cdots F_{n_{j-1}}) \subset K_j .\]
Thus, using~\eqref{eq:cell-process-folner}, we get that
\[ \tilde M_j = m_\Gamma(g F_{n_j} \setminus F_{n_j}) + m_\Gamma(g^{-1} F_{n_j} \setminus F_{n_j}) \le 2m_\Gamma(K_j F_{n_j} \triangle F_{n_j}) \le \tfrac12 N_j \quad\text{so that}\quad \tilde{N}_j \ge \tfrac 12 N_j.\]
Note that $|C_j(\tilde v) \cap C_j(\tilde w)|$ and $|C_j(\tilde v) \triangle C_j(\tilde w)|$
are independent random variables (conditionally on $(Z^1,\dots,Z^{j-1})$) whose distributions are $\text{Bin}(\tilde N_j,\frac{1}{N_j})$ and $\text{Bin}(\tilde M_j,\frac{1}{N_j})$, respectively.
Thus, given $(Z^1,\dots,Z^{j-1})$, the probability that 
\[|C_j(\tilde v) \cap C_j(\tilde w)| = 1 \qquad\mbox{and}\qquad |C_j(\tilde v) \triangle C_j(\tilde w)|=0\]
is
\[ \tfrac{\tilde N_j}{N_j}\left(1-\tfrac{1}{N_j}\right)^{\tilde{M}_j+\tilde N_j -1} \ge \min_{n \ge 1} \tfrac12(1-\tfrac{1}{n})^{3n}  =: c > 0 .\]
This yields a uniform lower bound on the conditional probability that $Z^j_v = Z^j_w$ given $(Z^1,\dots,Z^{j-1})$. Thus, the (unconditional) probability that $Z^j_v \neq Z^j_w$ is exponentially small in $j$. In particular, almost surely, $Z^j_v = Z^j_w$ for large enough $j$. This shows that there is a single eventual cell.
\end{proof}

\subsection{Finitary cycle-free permutations}

In this section, we give the main construction towards establishing \cref{prop:cycle-free-permutation} about the existence of a finitary cycle-free permutation and $\Z$-type total order, assuming the existence of a finitary cell process.

\begin{lemma}\label{lem:cycle-free-permutation}
Let $\Gamma < \mathit{Perm}(\V)$ be a semi-nice unimodular permutation group. Let $Y$ be a process that admits a total order $\prec$ on $\V$ as a finitary factor and that also admits a finitary cell process $\rho$ having infinite eventual cells. Then $Y$ admits a cycle-free permutation of $\V$ as a finitary factor. Furthermore, if the cell process has a single eventual cell, then the permutation almost surely has a single orbit.
\end{lemma}
\begin{proof}
For $v \in \V$ and $j \in \N$, let $C^j_v := \{ w \in \V~|~ \rho_j(Y)_w = \rho_j(Y)_v\}$ denote the cell of $v$ at level~$j$. The definition of a finitary cell process ensures that this set is finite, and furthermore that $Y \mapsto C^j$ is a finitary factor map.

Let $C^\infty_v$ be the eventual cell of $v$. We claim that there is no $\prec$-maximal element in $C^\infty_v$. To see this, consider the function that transports a unit of mass from $u$ to $v$ whenever $u \in C^\infty_v$ and $v$ is the $\prec$-maximal element of $C^\infty_v$. By unimodularity, the expected mass in is equal to the expected mass out, which is clearly at most 1. Since $\rho$ has infinite eventual cells, $|C^\infty_v|=\infty$ almost surely. In particular, on the event that $v$ is the $\prec$-maximal element of $C^\infty_v$, the total mass into $v$ is infinite. We thus conclude that this occurs with probability zero.

We now define a cycle-free permutation $S$ of $\V$.
We define $S$ by induction on $j$, so that after the $j$-th step, for each $v$, for all but one $u \in C^j_v$ (which we denote by $u^j_v$), the successor $S(u)$ of $u$ is defined and belongs to $C^j_v$, and similarly for the predecessor $S^{-1}(w)$ (with the element whose predecessor is undefined denoted by $w^j_v$).
We begin with $j=0$ for which we assume that $C^0_v=\{v\}$ so that there is nothing to define. Now fix $j \ge 1$ and $v \in \V$, and suppose that $C^j_v$ consists of cells $C_1,\dots,C_k$ of $C^{j-1}$, enumerated in such a way that $v_1 \prec v_2 \prec \cdots \prec v_k$, where $v_i$ is the $\prec$-minimum in $C_i$. For $1 \le i <k$, we set the successor of $u^{j-1}_{v_i}$ to be $w^{j+1}_{v_i}$ (the use of $v_i$ here is just for convenience; any element in $C_i$ would do). Then every $u \in C^j_v$ has a well defined successor, except for $u^j_v=u^{j-1}_{v_1}$, and every $w \in C^j_v$ has a well defined predecessor, except for $w^j_v=w^{j-1}_{v_k}$.
This completes the definition of $S$. It is straightforward that $S$ is a cycle-free permutation and that it is a finitary factor of $Y$.

Assume now that $\rho$ has a single eventual cell. Since, for every $j$ and $v$, the permutation $S$ cycles through all elements of $C^j_v$ (in the sense that there exists $u \in C^j_v$ such that $\{S^i(u)\}_{i=0}^{|C^j_v|-1} = C^j_v$), it follows that the permutation $S$ consists of a single orbit.   
\end{proof}

\subsection{Concluding \cref{prop:finite_cell_process}, \cref{prop:cycle-free-permutation} and \cref{thm:iid-isomorphic}}
\label{sec:concluding-cell-process}

So far, we have established \cref{lem:finite_cell_processI} and \cref{lem:finite_cell_processII} about the existence of finitary cell processes, and \cref{lem:cycle-free-permutation} about the existence of finitary cycle-free permutations and finitary $\Z$-type orders.
We show here how this yields \cref{thm:iid-isomorphic}. We then show how this theorem, together with the lemmas, establishes \cref{prop:finite_cell_process} and \cref{prop:cycle-free-permutation}.

\begin{proof}[Proof of \cref{thm:iid-isomorphic}]
Let $X$ and $Y$ be two countable-valued \iid\ processes having equal entropy and such that $X_v$ and $Y_v$ each take at least three values.
Let $0< \epsilon< 1$ be sufficiently small.
By splitting two atoms in $X$ into three atoms, we can obtain an \iid process $X'$ having an atom of mass exactly $\epsilon$ and the same entropy as $X$.
Similarly, by splitting two atoms in $Y$ into three atoms, we can obtain an \iid\ process $Y'$ having an atom of mass exactly $\epsilon$ and the entropy as $Y$.
 Thus, to show that $X$ and $Y$ are finitarily $\Gamma$-isomorphic, it suffices to show that any two \iid\ processes with equal entropy and an atom of equal weight are finitarily $\Gamma$-isomorphic.

We may thus assume without loss of generality that $X$ and $Y$ have an atom of equal weight. By renaming the symbols, we may further assume that they have a symbol $a$ of equal weight $p$. Let $Z$ be a $\{0,1\}$-valued \iid\ process of density $p$. By \cref{lem:finite_cell_processI}, \cref{lem:inv_total_order} and the first part of \cref{lem:cycle-free-permutation}, $Z$ admits a cycle-free permutation $S_Z$ of $\V$ as a finitary factor. Let $Z(X)$ and $Z(Y)$ denote the marker processes for $X$ and $Y$ given by the occurrences of $a$. We identify $Z(X)$, $Z(Y)$ and $Z$. Conditioned on $Z$, using the $\Z$-type order, we obtain a collection of $\Z$-processes: for each $v \in \V$ such that $Z_v=0$ (equivalently, $X_v \ne a$), we have the process $X^{(Z,v)}_n := X_{\tilde S_Z^n(v)}$, where $\tilde S_Z(v):= S_Z^{N_Z(v)}(v)$ and $N_Z(v)$ is the smallest positive integer such that $Z_{S_Z^{N_Z(v)}(v)}=0$. A routine Poincar\'e recurrence argument implies that $\tilde S_Z(v)$ is almost surely well defined when $Z_v=0$. Then each $X^{(Z,v)}$ is an \iid\ process (conditionally on $Z$), with a fixed deterministic distribution. Furthermore, given $Z$, for any subset of $F \subset  \{ v \in \V : Z_v=0\}$ having no pair of distinct elements in the same $S_Z$-orbit, the $\Z$-processes $(X^{(Z,v)})_{v \in F}$ are jointly independent.  We similarly define $Y^{(Z,v)}$, which is also an \iid\ process with the same entropy as $X^{(Z,v)}$. We may now apply the Keane--Smorodinsky finitary isomorphism (or, alternatively, for a self-contained proof, we may appeal to \cref{thm:relative_smorodinsky_k_dependent}) to each pair of these \iid\ $\Z$-processes.
\end{proof}

\begin{proof}[Proof of \cref{prop:finite_cell_process}]
The first part is just \cref{lem:finite_cell_processI}. The second part follows from \cref{lem:finite_cell_processII} and \cref{thm:iid-isomorphic}.
\end{proof}

\begin{proof}[Proof of \cref{prop:cycle-free-permutation}]
The first part follows from \cref{lem:finite_cell_processI}, \cref{lem:inv_total_order} and the first part of \cref{lem:cycle-free-permutation}. The second part follows from \cref{lem:finite_cell_processII}, \cref{thm:iid-isomorphic}, \cref{lem:inv_total_order} and the second part of \cref{lem:cycle-free-permutation}.
\end{proof}

\section{Completing the proof of \cref{thm:main}}\label{sec:proof_of_thm_main}

We now complete the proof of \cref{thm:main} by deducing it from the case of $\Z$-processes, which is  \cref{thm:realtive_efficient_finitary_coding_Z}.
The general idea is to use the existence of an equivariant finitary $\Z$-type total order on $\V$ guaranteed by \cref{prop:cycle-free-permutation} to move from $(\V,\Gamma)$-processes to $\Z$-processes, where we can apply \cref{thm:realtive_efficient_finitary_coding_Z}, and then return to $(\V,\Gamma)$-processes. This is rather straightforward when $\Gamma$ acts freely on $\V$, in which case one may ``record the increments'' obtained from the $\Z$-type total order (the increment at $v \in \V$ is the unique element $g \in \Gamma$ which sends $v$ to its successor). \cref{lem:from-Gamma-to-Z} below will allow us to handle the general case. Roughly speaking, it says that a $(\V,\Gamma)$-process which admits an equivariant finitary $\Z$-type total order on $\V$ can be finitarily identified with a certain $\Z$-process.

Let $W$ be a $(\V,\Gamma)$-process that admits a $\Z$-type total order $\prec_W$ on  $\V$ as a finitary factor.
Let $P_W \colon \V \to \V$ be the permutation of $\V$  defined by letting $P_W(v)$ be the $\prec_W$-successor of $v$. Given a factor $\tilde W$ of $W$, 
we fix $v_0 \in \V$ and define a $\Z$-process $\tilde W^\prec$ by
\[ \tilde W^{\prec}_n := \tilde W_{P_W^n(v_0)} .\]
Observe that whenever $\tilde W$ is a finitary factor of $W$, the map $W \mapsto \tilde W^\prec$ is also finitary.
For $v \in \V$, define $I^W(v_0,v) \in \Z$ by
\[ I^W(v_0,v) := n \mbox{ if } P_W^n(v_0)= v.\]
It can be directly verified that for every $g \in \Gamma$ and $n \in \mathbb{Z}$ we have
\begin{equation} \label{eq:prec_equiv}
g(\tilde W)^{g(\prec)}_n = \tilde W^\prec_{n + I^{W}(v_0,g^{-1}(v_0))},
\end{equation}
where $g(\prec) := \prec_{g(W)}$ and
 $g(\tilde W)^\prec_n :=  g(\tilde W)_{P_{g(W)}^n(v_0)}$.
In particular, for every $g \in \Gamma_{v_0}$ we have $g(\tilde W)^{g(\prec)} = \tilde W^\prec$, so $\tilde W^\prec$ is actually measurable with respect to $\Gamma_{v_0} W$.



Below is a simple auxiliary result that we will use: 
\begin{lemma}\label{lem:big_Z_process}
	Any aperiodic and ergodic $\mathbb{Z}$-process admits as a finitary factor an $\N$-valued process $R= (R_n)_{n \in \Z}$ such that almost surely $R_n > 2n$ for infinitely many positive $n$'s. 
\end{lemma}

The specific choice of $2n$ is what we need for our application, but is otherwise not important and can be replaced by any increasing function $g(n)$. Indeed, the process $R'_n := g(R_n)$ satisfies that almost surely $R'_n > g(n)$ for infinitely many positive $n$'s.

\begin{proof}
	Let $X$ be an aperiodic and ergodic $\mathbb{Z}$-process.
	We claim that $X$ has an $\N$-valued finitary factor $Y$ satisfying that $\sup_{n>0} Y_n = \infty$ almost surely. Indeed, if the process $X$ itself takes infinitely many values, this is trivial. Otherwise, choose a sequence $a = (a_n)_{n \in \Z}$ in the support of $X$, and define $Y_n := \max\{ k \ge 0 : X_{[n+1,n+k]}=a_{[1,k]}\}$. 
	
	We can now inductively define an increasing sequence of integers $(n_k)_{k=1}^\infty$ with $n_1=1$ such that
	\[ \Pr\left(\sum_{n=n_k}^{n_{k+1}}\1_{\{Y_n \ge n_k\}} \ge k\right) \ge 1 - \frac1k.\]
	Next define $f\colon\N\to\N$ by $f(j):=2n_{k+1}+1$ if $j \in [n_k,n_{k+1})$.
	Finally, define the process $R$ by $R_n:=f(Y_n)$, and observe that for every $k \in \N$, with probability at least $1- \frac 1k$, there are at least $k$ indices $n \in [n_k,n_{k+1}]$ with $R_n > 2n_{k+1} \ge 2n$.
\end{proof}

\begin{lemma}\label{lem:from-Gamma-to-Z}
Let $\Gamma < \mathit{Perm}(\V)$ be a semi-nice permutation group. Let $W$ be a countable-valued \iid\ $(\V,\Gamma)$-process 
that admits a $\Z$-type total order $\prec_W$ on $\V$ as a finitary factor. 
Then $W$ admits a finitary factor $\tilde W$ taking values in a discrete countable set such that  $\Gamma_{v_0} W$ is a topo-finitary function of $\tilde W^\prec$.
\end{lemma}

In the statement above by ``$\Gamma_{v_0} W$ is a topo-finitary function of $\tilde W^\prec$'' we mean that for every finite set $F \subset \V$ almost surely there exists a random $N \in \N$ such that $\{g(W)|_F : g \in \Gamma_{v_0}\}$ is uniquely determined by $(\tilde W^\prec_n)_{|n| \le N}$. In the case where $\Gamma$ acts freely on $\V$, we  could simply take $\tilde W_v = (W_v,g_v)$ where $g_v$ is the unique element of $\Gamma$ such that $g_v(v)$ is the $\prec_W$-successor of $v$.

\begin{proof}

Note that if $W$ is the trivial process the conclusion of the lemma is also trivial, so we can assume $W$ is non-trivial.
Had we allowed $\tilde W$ to take values in an uncountable Polish space, we could have chosen 
$\tilde W $ to be $(\Gamma_{v,v_0} W)_{v \in \V}$.  However this would yield a process taking values in a non-discrete space, which we do not want. Instead, we will record at each $v$ a finite approximation of $\Gamma_{v,v_0} W$, in such a way that the information at other $w$'s allows to reconstruct $\Gamma_{v,v_0}W$ entirely.
Here is the precise description:

Applying \cref{lem:big_Z_process} to the $\mathbb{Z}$-process $W^\prec$, we obtain a $\Z$-process $\tilde R$ as a finitary factor so that $\tilde R_n > 2n$ for  infinitely many positive $n$'s. Define a $(\V,\Gamma)$-process $R$ by $R_v = \tilde R_{I^W(v_0,v)}$ (so that $R^\prec = \tilde R$). Then  $R$ is a finitary factor of $W$, and almost surely $R_{P_W^{n}(v_0)} > 2n$ for infinitely many positive $n$'s.
 

Define $\tilde W$ by
\[ \tilde W_{v} := \big\{ \psi_{R_v}(g(W)) : g \in \Gamma_{v,v_0} \big\},\]
where
\[ \psi_r(W) := \left(W_{P_W^j(v_0)},P_W^j(v_0)\right)_{ |j| \le r} .\]
It can be directly verified that the map $W \mapsto \tilde W$ is equivariant.
Because $R$ and $P_W$ are finitary factors of $W$, it also follows that $\tilde W$ is indeed a finitary factor of $W$.

It remains to show that $\Gamma_{v_0} W$ is a topo-finitary function of $\tilde W^\prec$, that is, that almost surely for any finite $F \subset \V$ there exists a random $N \in \N$ such that $\{ g(W)|_F :g \in \Gamma_{v_0}\}$ is uniquely determined by $(\tilde W_n^{\prec})_{|n| \le N}$. Since, almost surely, $\V =\{ P_W^j(v_0) : j \in \Z \}$ and $R_{P_W^N(v_0)} > 2N$ for infinitely many $N$'s, it in turn suffices to prove that for any such $N$, the set $\{ g(W) \mid_{\{P_{g(W)}^n(v_0)\}_{|n| \le N }} :g \in \Gamma_{v_0}\}$ is uniquely determined by $\tilde W^{\prec}_N$. Indeed, this is straightforward from the definitions.
\end{proof}

We are now ready to prove \cref{thm:main}.

\begin{proof}[Proof of \cref{thm:main}]
Let $\Gamma < \mathit{Perm}(\V)$ be a nice amenable group.
Let $X$ be a countable-valued $(\V,\Gamma)$-process which is a finitary $\Gamma$-factor of an \iid\ process $Y$, and let $\epsilon >0$. Let $W$ be an \iid\ $(\V,\Gamma)$-process independent of $Y$, having entropy at most $\epsilon/2$. By \cref{prop:cycle-free-permutation}, there exists a random invariant $\Z$-type order $\prec_W$ on $\V$ which is a finitary factor of $W$.
As before, for $v \in \V$, let $P_W(v)$ be the $\prec_W$-successor of $v$. 

Apply \cref{lem:from-Gamma-to-Z} to $W$ to obtain a countable-valued $(\V,\Gamma)$-process $\tilde W$ as a finitary factor of $W$ such that $\Gamma_{v_0} W$ is a topo-finitary function of $\tilde W^\prec$.
Then $(X^\prec,Y^\prec,\tilde W^\prec):=(X,Y,\tilde W)^\prec$ is a $\Z$-process such that $Y^\prec$ is an \iid\ process independent of $\tilde W^\prec$.
Let us show that $X^\prec$ is a finitary $\Z$-factor of $(Y^\prec,\tilde W^\prec)$.
Because $X$ is a finitary factor of $Y$, almost surely there exists a finite  $F \subset \V$  such that $Y_F$ determines $X_{v_0}$. 
Because the map from $Y$ to $X$ is $\Gamma$-equivariant, the map from $Y$ to $X_{v_0}$ is $\Gamma_{v_0}$-invariant. Replacing $F$ by  $\Gamma_{v_0} F$, we can assume that $F$ is $\Gamma_{v_0}$-invariant. So almost surely there exists a random $\Gamma_{v_0}$-invariant set $F$ so that $\{ g(Y)_F : g \in \Gamma_{v_0} \} =\Gamma_{v_0} (Y_F)$ already determines $X_{v_0}$.
Since $\tilde W^\prec \mapsto  \Gamma_{v_0}W$ is topo-finitary, for every $n \in \N$ almost surely there exists $N \in \N$ such that $\tilde W^\prec_{[-N.N]}$ determines  $\{ g(W)_{I_{n,g(W)}}:g \in \Gamma_{v_0}\}$, where $I_{n,W}:= \{ v \in \V: |I^W(v_0,v)|< n\}$. 
For every finite set $F$, almost surely there exists $n \in \N$ such that $\{ g(W_{I_{n,g(W)}}):g \in \Gamma_{v_0}\}$ determines 
the possible $\prec_{g(W)}$-orderings of the elements of $F$, as $g$ ranges over $\Gamma_{v_0}$.
Now observe that 
$ \{g(Y)_F :~g \in \Gamma_{v_0}\} $ is determined by $Y^\prec$ together with the possible $\prec_{g(W)}$-orderings of the elements of $F$, as $g$ ranges over $\Gamma_{v_0}$.
We conclude that
 $X^\prec_0$ is a finitary  function of $(Y^\prec,\tilde W^\prec)$, and by  \eqref{eq:prec_equiv} actually $X^\prec$ is a finitary $\mathbb{Z}$-factor of $(Y^\prec,\tilde W^\prec)$.
  Thus, $X^\prec$ is a finitary factor of \iid\ relative to $\tilde W^\prec$ so that \cref{thm:realtive_efficient_finitary_coding_Z} yields a $\Z$-process $\tilde X'$ with $H(\tilde X'_0)< \epsilon/2$ such that $(X^\prec,\tilde X')$ is finitarily isomorphic to an \iid\ $\Z$-process relative to $\tilde W^\prec$.

Now consider the $(\V,\Gamma)$-process $X'$ defined by 
\[
X'_{P^n_W(v_0)} := \tilde X'_{n},\qquad n \in \Z.
\]
We have that $(X')^\prec = \tilde X'$.
Because $P_W$ (and hence $\prec_W$) is a finitary function of $W$, it follows that $(X,X')$ is finitary isomorphic to an \iid\ $(\V,\Gamma)$-process relative to $W$. We conclude that $(X,X',W)$ is finitary isomorphic to an \iid\ $(\V,\Gamma)$-process. The ``moreover'' part of the theorem now follows from \cref{thm:iid-isomorphic}.
\end{proof}

\section{Further remarks and open problems}
\label{sec:conclusion}
We conclude with some remarks open questions and possible further directions that are related to our results.
 
\subsection{Perfectly efficient finitary codings}

The following question asks about the possibility of taking $\epsilon=0$ in \cref{thm:main,thm:amenable_groups_efficient_finitary_factors_of_iid}. For simplicity, we state in for finite-valued $\Z$-processes.
\begin{quest}
Let $X$ be a finite-valued process which is a finitary factor of an \iid\ process.
Does there exist a zero-entropy process $X'$ such that $(X,X')$ is finitarily isomorphic to an \iid\ process? Or is $X$ at least a finitary factor of an \iid\ process with the same entropy as $X$?
\end{quest}
By Gabor's result \cite{gabor2019finitary}, a finite-valued process $X$ which is a finitary factor of \iid\ need not be finitarily isomorphic to \iid\ in general, so that the process $X'$ cannot in general be taken as a trivial process, or even as an independent \iid\ process. One may still wonder if there are additional ``natural conditions'' that imply that $X$ is finitarily isomorphic to an \iid\ process.

\subsection{Coding radius}\label{sec:coding-radius}
Whenever $\V$ admits a  $\Gamma$-invariant metric in which balls are finite (typically the graph metric  when $\V$ is the vertex set of a locally finite graph $G$ and $\Gamma$ is a subgroup of the graph automorphisms of $G$), the \textbf{coding radius} of a finitary map $\pi$ from a $(\V,\Gamma)$-process $X$ to a $(\V,\Gamma)$-process $Y$ is a variable indicating indicating the smallest $R$ such that the restriction of $X$ to a ball of radius $R$ around $v_0 \in \V$  determines $Y_{v_0}$.

For the critical Ising model on $\Z^d$, it is known that the unique Gibbs measure is a finitary factor of \iid, but that the coding radius \emph{cannot} have a finite $d$-th moment~\cite{van1999existence}.
For the high-temperature Ising model on $\Z^d$, it was shown in~\cite{van1999existence} that the unique Gibbs measure is a finitary factor of \iid\ with a coding radius having exponential tails, and it was asked whether it is also a finitary factor of a \emph{finite-valued} \iid\ process with a coding radius having finite $d$-th moment~\cite[Question~2]{van1999existence}. This was answered affirmatively in~\cite{spinka2018finitaryising} with a coding radius having stretched-exponential tails. Our results further show that it is a finitary factor of a finite-valued \iid\ process whose entropy is only slightly larger than that of the Gibbs measure. The following asks whether the two properties can be obtained simultaneously:

\begin{quest}
Is the high-temperature Ising model on $\Z^d$ a finitary factor of a finite-valued \iid\ process whose entropy is only slightly larger than that of the Gibbs measure and with a coding radius having  stretched-exponential tails?
\end{quest}

More generally, one may ask for versions of our results which take into consideration the coding radius, both in the assumptions and the conclusions. For example, the following is (one possible variant) of a such version of \cref{thm:amenable_groups_efficient_finitary_factors_of_iid} for $\Z^d$-processes:

\begin{quest}
Let $X$ be a finite-valued $\Z^d$-process which is a finitary factor of \iid\ with a coding radius having exponential tails. Is $X$ also a finitary factor of a finite-valued \iid\ process with entropy only slightly larger than $X$ and with a coding radius having exponential tails?
\end{quest}

Let us mention that the corresponding question where ``exponential tails'' in replaced with ``bounded'' is false already for $\Z$-processes (even if we were to drop the entropy constraint), as the following simple example demonstrates: take $X_n:=\1_{\{U_n<U_{n+1}\}}$, where $U$ is an \iid\ process consisting of uniform $[0,1]$ random variables. Clearly, $X$ is a block factor of $U$ (equivalently, a finitary factor with bounded coding radius), but $\Pr(X_0=\cdots=X_k) \approx 1/k!$ decays super-exponentially and this is easily seen to be an obstruction for being a block factor of an \iid\ process having even one atom.

\subsection{Finitely dependent  $(\V,\Gamma)$-processes}
In \cite{spinka2019finitely}, it was shown that a finitely dependent $(\V,\Gamma)$-process is a finitary factor of an \iid\ process, where $\Gamma$ is a nice amenable permutation group of automorphisms of an amenable graph whose vertex set is $\V$.
 It seems plausible that for a natural notion of  ``finitely dependent'' for $(\V,\Gamma)$-processes over nice amenable permutation groups, it is the case that such processes are finitarily isomorphic \iid. The methods developed in this paper can be used to obtain some partial results in this direction, but we currently do not know if this holds in full generality.

\begin{quest}
	Are finitely dependent processes over nice amenable permutation groups finitarily isomorphic to \iid?
\end{quest}

The case $\V=\Gamma=\Z^d$ has been  posed as an open problem in \cite{spinka2019finitely}.
\subsection{$(\V,\Gamma)$-processes over nice but non-amenable groups}
The conclusion of \cref{thm:amenable_groups_efficient_finitary_factors_of_iid} and the `in particular' part of \cref{thm:main} are known to hold in some situations without the amenability assumption. For instance, the case when $\V$ is the set of vertices of the  the $d$-regular tree for $d >2$, and $\Gamma$ is the group of automorphisms of the tree basically follows from the methods in \cite{ball2005factors}, where it is shown that every \iid\ process is a  factor of the full 2-shift. Lewis Bowen proved that for any non-amenable countable group $\Gamma$, all non-trivial \iid\ processes factor onto each other~\cite{bowen2019finitaryinterlacments}. The construction in  \cite{bowen2019finitaryinterlacments} does not seem to give a finitary factor in general.

\begin{quest}
	Is it possible to remove the amenability assumption on the permutation group $\Gamma$ in the statements of \cref{thm:amenable_groups_efficient_finitary_factors_of_iid,thm:main}?
\end{quest}

The statements which do not involve finitary isomorphisms, namely, that of \cref{thm:amenable_groups_efficient_finitary_factors_of_iid} and the `in particular' part of \cref{thm:main}, would follow from a positive answer to the following:

\begin{quest}
	Let $\Gamma$ be a nice non-amenable permutation group. Is every non-trivial countable-valued \iid\ process a finitary factor of every other such process?
\end{quest}

\subsection{Which permutation groups are finitarily Ornstein?}

As mentioned in the introduction, it follows from \cite{seward2018bernoulli} that any countable permutation group $\Gamma$ acting transitively and freely on $\V$ is finitarily Ornstein (our results complement the case of countable-valued \iid\ process).  We ask:
\begin{quest}
	Is every nice permutation group finitarily Ornstein?
\end{quest}

In view of \cref{thm:iid-isomorphic}, the problem amounts to showing that any $2$-valued \iid\ $(\V,\Gamma)$-process is finitarily isomorphic to some equal entropy \iid\ $(\V,\Gamma)$-process taking more than two values. This seems to be open even in the case where $\Gamma$ is automorphism group of the Cayley graph of $\Z^d$.

\subsection{Which permutation groups are Kolmogorov?}

We say that a group of permutations $\Gamma$ of $\V$ is \textbf{(finitarily) Kolmogorov} if any two finitarily isomorphic \iid\ $(\V,\Gamma)$-processes have equal entropy.
There is the question of equal entropy being a necessary condition for (finitarily) isomorphism: 

\begin{quest}\label{quest:entropy_invariant}
	Which permutation groups $\Gamma$ are (finitarily) Kolmogorov?
\end{quest}

We note that being Kolmogorov is stronger than being finitarily Kolmogorov, and that each of the two properties is preserved when moving to a larger permutation group.
In the case where $\Gamma$ is a cyclic group generated by a transitive permutation of $\V$, this question, which goes back to von Newmann in the early years of measurable dynamics, was answered by Kolmogorov, who famously introduced dynamical entropy of general probability measure-preserving transformations. 
A variation of Kolmogorov's original argument shows that every semi-nice unimodular amenable permutation group is Kolmogorov (see \cref{sec:entropy}).
Already for classical $\Gamma$-processes, it is a major open problem to determine whether isomorphic \iid\ $\Gamma$-processes have equal entropy for any countable group $\Gamma$. See \cite{seward2019kriegerII} for remarkable (hypothetical) consequences of a solution to this problem.
Simple examples of permutation groups which are not Kolmogorov (or even finitarily Kolmogorov) are finite permutation groups of a countably infinite set $\V$.
There are also examples of non-Kolmogorov permutation groups in which all orbits are infinite. For instance, when $\V = \mathbb{Z}^2$ and $\Gamma \cong \mathbb{Z}$ is the group generated by a translation by a single non-zero element of $\mathbb{Z}^2$, any pair of non-trivial $(\V,\Gamma)$-processes are finitarily isomorphic (they are all isomorphic to an infinite-entropy Bernoulli shift over $\mathbb{Z}$).
We are not aware of any  non-Kolmogorov permutation groups which are semi-nice and unimodular.

\bibliographystyle{amsplain}
\bibliography{library}
\end{document}